%% file: main.tex
\documentclass[11pt,letterpaper]{amsart}
\input{head}

\makeatletter
\@namedef{subjclassname@2020}{\textup{2020} Mathematics Subject Classification}
\makeatother

\begin{document}

\title[$\delta$-stable minimal hypersurfaces]{On $\delta$-Stable Minimal Hypersurfaces in $\mathbb{R}^{n+1}$}

\author{Han Hong}
\address{Department of Mathematics and statistics \\ Beijing Jiaotong University \\ Beijing \\ China, 100044}
\email{hanhong@bjtu.edu.cn}

\author{Haizhong Li}
\address{Department of Mathematical Sciences, Tsinghua University, 100084, Beijing, China}
\email{lihz@tsinghua.edu.cn}

\author{Gaoming Wang}
\address{Yau Mathematical Sciences Center, Tsinghua University, 100084, Beijing, China}
\email{gmwang@tsinghua.edu.cn}

\begin{abstract}
In this paper, we extend several results established for stable minimal hypersurfaces to $\delta$-stable minimal hypersurfaces. These include the regularity and compactness theorems for immersed $\delta$-stable minimal hypersurfaces in $\mathbb{R}^{n+1}$ when $n \geq 3$ and $\delta > \frac{n-2}{n}$, as well as the $\delta$-stable Bernstein theorem for $n=3$ and $n=4$ for properly immersion. The range of $\delta$ is optimal, as the $n$-dimensional catenoid in $\mathbb{R}^{n+1}$ is $\frac{n-2}{n}$-stable.
\end{abstract}

\maketitle

\section{Introduction}
An immersed minimal hypersurface $M^n$ in $\mathbb{R}^{n+1}$ is said to be \textit{stable} if 
\[\int_M |\nabla \varphi|^2-|A|^2\varphi^2\geq 0,\]
for all $\varphi\in C_0^\infty(M)$ where $A$ is the second fundamental form of $M$.
Recent progress has significantly improved the understanding of stable minimal hypersurfaces in $\mathbb{R}^{n+1}$.
The stable Bernstein problem, asserting that any two-sided, complete, noncompact, stable minimal hypersurfaces in $\mathbb{R}^{n+1}$ must be flat for $2\le n\le 6$, has been solved up to $n=5$ by employing the warped $\mu$-bubble technique.

For $n\ge 3$ case, the readers are referred to Chodosh-Li \cite{chodoshliR4anisotropic} for $n=3$, Chodosh-Li-Minter-Stryker \cite{chodoshliR5} for $n=4$, and Mazet \cite{mazet} for $n=5$.
The only remaining case is $n=6$. It is worth mentioning in dimension $n=3$, Chodosh-Li \cite{chodoshliR4} and Giovanni-Paolo-Alberto \cite{catino} have provided two alternative proofs.
In $\mathbb{R}^3$ the plane is the only stable complete minimal surface \cite{Fischer-Colbrie-Schoen-The-structure-of-complete-stable,doCarmo-Peng, Pogorelov-stable, Ros-onesided-minimalsurface} and the catenoid is the only two-sided embedded complete minimal surface with index one \cite{lopezros,chodoshmaximojdg}. While higher-dimensional catenoids have index one \cite{tamzhou}, it remains unsolved whether they are the only embedded complete minimal hypersurfaces in $\mathbb{R}^{n+1}$, $n\geq 3$ with index one \cite{lichaoindexestimate}.

Assuming the volume growth condition, the Bernstein theorem for stable immersed minimal hypersurfaces in $\mathbb{R}^{n+1}$ has already been proved by Schoen-Simon-Yau \cite{Schoen1975curvature} for $3\le n\le 5$ and recently by Bellettini \cite{bellettini} for $3\leq n\leq 6$ using a new idea.
Previous to Bellettini's work, if we further require the minimal hypersurfaces to be embedded, the corresponding Bernstein theorem was proved by Schoen-Simon \cite{SchoenSimon1981Regularity} for $3\le n\le 6$.
In fact, Schoen-Simon established a regularity and compactness theorem for embedded stable minimal hypersurfaces, assuming the $\mathcal{H}^{n-2}$ Hausdoff measure of the singular set is zero, which is a stronger result than the Bernstein theorem.
The regularity and compactness results have been extended to the case where the classical singularities are absent by Wickramasekera \cite{Wickramasekera2014Regularity}.

Our primary goal is to extend these results to $\delta$-stable minimal hypersurfaces in $\mathbb{R}^{n+1}$ for $n\ge 3$.
Given $\delta>0$, a minimal hypersurface $M$ in $\mathbb{R}^{n+1}$ is said to be \textit{$\delta$-stable} if
\begin{equation}\int_M |\nabla \varphi|^2-\delta|A|^2\varphi^2\geq 0\label{eq:deltaStability}\end{equation}
for any $\varphi\in C_0^\infty(M)$.
Hyperplanes are stable and $\delta$-stable for all $\delta\in (0,1)$.
The $n$-dimensional catenoid in $\mathbb{R}^{n+1}$ is $\frac{n-2}{n}$-stable since there exists a positive function $u$ such that $\Delta u+\delta |A|^2u=0$ \cite{tamzhou} and the characterization of $\delta$-stability is equivalent to the existence of a positive kernel of the operator $\Delta+\delta |A|^2$ \cite{Fischer-Colbrie-On-complete-minimal}. 

In dimension two, Colding and Minicozzi \cite{Colding-Minicozzi-The-space-of-emb...II} raised the above concept to study the total curvature of embedded minimal discs. Kawai \cite{kawai} proved that a complete $\delta$-stable minimal surface in $\mathbb{R}^3$ for $\delta>\frac{1}{8}$ must be a plane by studying the operator $-\Delta+\delta K$ on surfaces.
Meeks-P\'erez-Ros \cite{meeksperezros} conjectured that a complete embedded $\delta$-stable minimal surface in a complete flat three-manifold is totally geodesic (flat), and they proved this conjecture under the assumption of finite genus.
In higher dimensions, Tam and Zhou \cite{tamzhou} proved that an $\frac{n-2}{n}$-stable complete immersed two-sided minimal hypersurface in $\mathbb{R}^{n+1}$ satisfying the following curvature decay condition
\begin{equation}
    \lim_{R\rightarrow \infty}\frac{1}{R^2}\int_{M\cap B^{n+1}_{2R}(0)\setminus B^{n+1}_{R}(0)}|A|^{\frac{2(n-2)}{n}}=0,\label{tamzhoutotalcurvaturegrowthcondition}
\end{equation}
is either a hyperplane or a catenoid.
Fu \cite{Fu2009deltaStable} proved that a complete, two-sided, $\delta$-stable minimal hypersurface in $\mathbb{R}^{n+1}$ with $\delta>\frac{(n-1)^2}{n^2}$ and finite total curvature is flat.
Specifically, for $n=3$, Chodosh and Li \cite{chodoshliR4} proved that a complete, two-sided, $\delta$-stable minimal hypersurface in $\mathbb{R}^{4}$, with $\delta=\frac{2}{3}$ or $\delta>\frac{1}{2}$, which is simply connected and has one end, must also be flat.

The concept of $\delta$-stability also relates to the stability  of anisotropic area functional. In fact, we consider $\Phi(M)=\int_M\Phi(\nu(x))$ where $\Phi(\cdot)$ is a positive $1$-homogeneous smooth function defined on unit sphere. The critical point of $\Phi$ is called $\Phi$-stationary and it is said to be $\Phi$-stable if the second derivative of $\Phi$ with respect to compact variation is nonnegative. We often take $\Phi$ to be elliptic, that is $I\leq D^2\Phi\leq cI$ for a constant $c$. The $\Phi$-stability implies that
\[\int_M|\nabla \phi|^2\geq \int_M\Lambda|A|^2\phi^2.\]
The constant $\Lambda$ depends on the ellipticity of $\Phi$ and note that $\Lambda\geq \frac{1}{c}$ (see \cite{chodoshliR4anisotropic} for discussion). In fact, Chodosh-Li \cite{chodoshliR4anisotropic} showed that $\frac{1}{\sqrt{2}}$-stability implies flatness of $M$.

In this paper, we focus on the regularity and compactness theorem for $\delta$-stable immersed minimal hypersurfaces.
This leads to the Bernstein theorem for $\delta$-stable minimal hypersurfaces in $\mathbb{R}^{n+1}$ for $n\ge 3$ and $\delta>\min \{ \frac{n-2}{n},\frac{(n-2)^2}{4(n-1)} \}$ under the Euclidean volume growth condition.
Subsequently, we aim to extend the Bernstein theorem to $\delta$-stable minimal hypersurfaces in $\mathbb{R}^{n+1}$ for $n=3,4$, without assuming the volume growth condition.

\subsection{Regularity and Compactness Theorem of $\delta$-Stable Minimal Hypersurfaces}

We describe our first main result.

For any $\delta>0$, we define:
\[
	n_\delta:=\inf\left\{ n \in \mathbb{N}: n\ge 2 \text{ and } \frac{(n-2)^2}{4(n-1)}\ge \delta \right\}.
\]

\begin{theorem}
	[Regularity and Compactness Theorem]
	Let $n\ge 3$ and $\delta>\frac{n-2}{n}$.
	Suppose $\left\{ M_k \right\}$ is a sequence of two-sided, $\delta$-stable minimal hypersurfaces immersed in $B^{n+1}_4(0)$ with
	\[
		0 \in \bar{M}_k, \text{ for each }k\quad \text{and}\quad \sup_{k}\mathcal{H}^n(M_k\cap B^{n+1}_4(0))<+\infty,
	\]
        and the singular set of $M_k$ satisfies
        \[
            \sup_{k}\mathrm{dim}(\bar{M}_k\backslash M_k \cap B^{n+1}_4(0))< n-4+\frac{4}{n},
        \]
        where $\bar{M}_k$ denotes the closure of $M_k$.
	Then there exists a stationary varifold $V$ of $B^{n+1}_4(0)$ and a closed subset $S\subset \mathrm{spt}\|V\|\cap B^{n+1}_{\frac{1}{4}}(0)$ with $S$ empty if $3\le n\le n_\delta-1$, $S$ discrete if $n=n_\delta$, and $\mathcal{H}^{n-n_\delta+\gamma}(S)=0$ for every $\gamma>0$ if $n\ge n_\delta+1$ such that,
	up to a subsequence, $M_k$ converges to $V$ in the varifold sense and $\mathrm{spt}\|V\|\backslash S\cap B^{n+1}_{\frac{1}{4}}(0)$ is a two-sided, $\delta$-stable minimal hypersurface immersed in $B^{n+1}_{\frac{1}{4}}(0)$.
	\label{thm_regularity_and_compactness_theorem}
\end{theorem}
Readers are directed to Section \ref{sec:notation} for the relevant notations.

This theorem generalizes the regularity and compactness theorem due to Schoen-Simon \cite{SchoenSimon1981Regularity} to the immersed case and $\delta$-stable case.
For the usual stable case, $\delta=1$, it was conjectured by Bellettini \cite{bellettini} as follows:
\begin{conjecture}
	The class of branched two-sided stable minimal $n$-dimensional immersions with the singular set of locally finite $(n-2)$-measure is compact under varifold convergence.
\end{conjecture}
Note that Wickramasekera \cite{Wickramasekera2008regularityCompactMult2} verified this conjecture for the case of multiplicity-2.
However, the general case, even for the case of codimension-7 singular set, remained open until our work.
Our theorem confirms the above conjecture when the singular set has codimension greater than $4-\frac{4}{n}$.
\begin{corollary}
	The class of two-sided, stable minimal $n$-dimensional immersions with singular set of locally finite $\bar{n}$-measure with $0<\bar{n}<n-4+\frac{4}{n}$ is compact under the varifold convergence.
\end{corollary}

If we choose $\delta>\frac{(n-2)^2}{4(n-1)}$, which implies $n_\delta>n$, then the singular set of $M$ is empty, and we can obtain the following Generalized Bernstein Theorem and the curvature estimate in the $\delta$-stable case.

\begin{corollary}[Generalized Bernstein Theorem]
	Let $n\geq 3$ and $\delta>\max\{\frac{n-2}{n},\frac{(n-2)^2}{4(n-1)}\}$. Suppose $M^n$ is a complete, two-sided, $\delta$-stable minimal hypersurface in $\mathbb{R}^{n+1}$ satisfying the Euclidean volume growth condition
        \begin{equation}
            \mathcal{H}^{n}(M\cap B^{n+1}_R(0))\le \Lambda R^n
            \label{eq:areaGrowthCondition}
        \end{equation}
	for some $\Lambda \in (0,+\infty)$.
	Then $M$ is flat.
	\label{thm:mainAreaGrowth}
\end{corollary}
\begin{corollary}[Curvature Estimate]
	Let $n\ge 3$ and $\delta> \max \{ \frac{n-2}{n},\frac{(n-2)^2}{4(n-1)} \}$.
	Assume $M$ is a two-sided, $\delta$-stable minimal hypersurface immersed in $B^{n+1}_4(0)$ with $0 \in M$ satisfying $\mathcal{H}^n(M \cap B^{n+1}_4(0))\le \Lambda$ for some $\Lambda \in (0,\infty)$.
	Then there exists a constant $C=C(n,\Lambda,\delta)$ such that
	\[
		\sup_{x \in B^{n+1}_{\frac{1}{4}}(0)}|A_M|\le C
	\]
	where $A_M$ is the second fundamental form of $M$.
	\label{thm:mainCurvatureEst}
\end{corollary}

These two results extend the works of Schoen-Simon-Yau \cite{Schoen1975curvature} and Bellettini \cite{bellettini} in two aspects. First, the lower bound of $\delta$ is improved to $\frac{n-2}{n}$ for $3 \leq n \leq 5$ and $\frac{4}{5}$ for $n=6$.
Second, it considers $\delta$-stability for $n \geq 7$, which is also new, to the best of the authors' knowledge.
The range of $\delta$ is sharp since the $n$-dimensional catenoid is $\frac{n-2}{n}$-stable, and Simons-type cone $C_{p,q}$ defined by
\[
	C_{p,q}= \{ (x,y) \in \mathbb{R}^{p+1}\times \mathbb{R}^{q+1}:q|x|^2=p|y|^2 \},
\]
is $\frac{(n-2)^2}{4(n-1)}$-stable for $1\le p\le n-2$, $q=n-p-1$ (cf. 
\cite{Simons1968minimalCone, Li2020bernsteinCone}).

To prove Theorem \ref{thm_regularity_and_compactness_theorem}, we also need to classify all $\delta$-stable minimal cones in $\mathbb{R}^{n+1}$ for $\delta > \frac{(n-2)^2}{4(n-1)}$, $n \geq 2$, which is of independent interest.
\begin{proposition}
	\label{thm:deltaCone}
	Suppose $M$ is a $\delta$-stable (immersed) minimal cone in $\mathbb{R}^{n+1}$ for $\delta>\frac{(n-2)^2}{4(n-1)}$ and $n\ge 2$.
	Then $M$ is flat.
\end{proposition}
The range of $\delta$ is sharp due to Simons-type cone.
Caiyan Li \cite{Li2020bernsteinCone} has already proved Proposition \ref{thm:deltaCone} for $\delta>\frac{2}{n}-\frac{1+(n-2)\sqrt{n}}{n-1}$, which is not sharp.

The rest of the proof of Theorem \ref{thm_regularity_and_compactness_theorem} follows the method of Bellettini \cite{bellettini}, with a Caccioppoli-type inequality, an iteration \`a la De Giorgi, and a tangent cone analysis. In particular, we show that by modifying the proof of Theorem 3 in \cite{bellettini}, we can establish the $\varepsilon$-regularity theorem for the second fundamental form (Theorem \ref{thm:epsilonRegularity}) for $\delta$-stable minimal hypersurfaces in $\mathbb{R}^{n+1}$ for $n\ge 3$ and $\delta>\frac{n-2}{n}$.

\subsection{$\delta$-Stable Bernstein Theorem in $\mathbb{R}^4$ and $\mathbb{R}^5$}%
\label{sub:_delta_stable_bernstein_theorem_in_r_4_and_r_5_}
The next main result in our paper addresses the Bernstein theorem for properly immersed $\delta$-stable minimal hypersurfaces in $\mathbb{R}^4$ and $\mathbb{R}^5$, without assuming the volume growth condition.

\begin{theorem}
    [$\delta$-stable Bernstein Theorem] \label{delta-stable}
    Let $n=3,4$ and $\delta>\frac{n-2}{n}$. 
    Suppose $M$ is a complete, two-sided, $\delta$-stable immersed minimal hypersurface in $\mathbb{R}^{n+1}$.
    If $M$ also satisfies one of the following conditions:
    \begin{enumerate}[\normalfont(a)]
        \item $\delta>\frac{n(n-2)}{4(n-1)}$; or \label{thm:itSSYdelta}
        \item $M$ is simply-connected and properly immersed with finite ends.
        \label{thm:itProper}
    \end{enumerate}
    Then $M$ is flat.
\end{theorem}

Since the $n$-dimensional catenoid is $\frac{n-2}{n}$-stable, the range for $\delta$ is optimal. This theorem improves some results regarding $\delta$-stability in \cite{chodoshliR4}.
\begin{remark}
    For \ref{thm:itProper} in Theorem \ref{delta-stable}, the requirement for properness is essential because we need to derive volume growth results using Euclidean balls rather than intrinsic geodesic balls so that Corollary \ref{thm:mainAreaGrowth} can be applied. Although we have flatness conclusion (see Proposition \ref{ssy-delta-stable} below)  if we assume $\mathcal{H}^n(B^M_R(p))\le CR^n$ for some $p\in M$ instead of \eqref{eq:areaGrowthCondition} in Corollary \ref{thm:mainAreaGrowth}, where $B^M_R(p)$ denotes the intrinsic geodesic ball on $M$,
we do not know whether this is true for the optimal range of $\delta$, i.e., $\delta>\frac{n-2}{n}$.
\end{remark}

Notice that $M$ stated in Corollary \ref{thm:mainAreaGrowth} is automatically proper. As mentioned in above remark, we need to prove a generalized Bernstein theorem for  $\delta$-stable minimal hypersurface in $\mathbb{R}^4$ and $\mathbb{R}^5$ with intrinsic volume growth condition.  This can be done by adopting the approach of Schoen-Simon-Yau \cite{Schoen1975curvature}, although this does not give us the optimal range of $\delta.$

\begin{proposition}\label{ssy-delta-stable}
     Let $n=3,4$ and $\delta>\frac{n(n-2)}{4(n-1)}$. Suppose $M$ is a complete, two-sided, $\delta$-stable minimal hypersurface in $\mathbb{R}^{n+1}$, $0\in M$, satisfying the intrinsic volume growth
     \[\mathcal{H}^{n}(M\cap B^{M}_R(0))\le \Lambda R^n\]
     for some $\Lambda\in(0,\infty)$. Then $M$ is flat.
     
\end{proposition}

\subsection{Idea of the proof of \ref{thm:itSSYdelta}}We show that a complete, two-sided, $\delta$-stable minimal hypersurface in $\mathbb{R}^{4}$ or $\mathbb{R}^5$ satisfies a curvature estimate: \textit{there exists a universal constant $C>0$ such that for any two-sided, $\delta$-stable  minimal immersion $M^n\rightarrow \mathbb{R}^{n+1}$ for $n=3,4$, \[|A_M|d_M(x,\partial M)\leq C.\]}

If $M$ is complete, then $M$ is flat. The proof of curvature estimate uses the standard point-picking argument (\cite{brianwhitenote}, see also \cite{chodoshliR4}). Indeed, let $n=3$ or $4$. If the conclusion of curvature estimate is not true, we could take a sequence of two-sided, $\delta$-stable minimal hypersurfaces $M_i$  and $p_i\in M_i$ such that\[|A_{M_i}|(p_i)d_{M_i}(p_i,\partial M_i)\longrightarrow \infty.\] We select $p_i \in M_i$ such that\[|A_{M_i}|(p_i)d_{M_i}(p_i,\partial M_i)\ge \frac{1}{2}\sup_{p \in M_i}|A_{M_i}|(p)d_{M_i}(p,\partial M_i).\]

By translating and rescaling, we can assume $p_i$ is the origin and $|A_{M_i}|(0)=1$. By the convergence  argument, we can choose a subsequence of $M_i$ such that it converges locally and smoothly to a complete, two-sided, $\delta$-stable minimal immersion $M_\infty\rightarrow \mathbb{R}^n$ with $|A_{M_\infty}|(0)=1$ and $|A_{M_\infty}|\leq 2$. Note that the limit could be nonproper minimal immersion. Consequently, the proof reduces to showing that a complete, two-sided, $\delta$-stable minimal hypersurface $M$ in $\mathbb{R}^n$ with bounded second fundamental form must be flat.

Since $\delta$-stability, completeness, two-sidedness, and boundedness of the second fundamental form  pass to the universal cover, we can assume that $M$ is simply-connected. Note that the universal cover of a proper immersion could be nonproper.
It is well-known that a stable minimal hypersurface  in Euclidean space has only one end, which is nonparabolic \cite{Cao-Shen-Zhu-infinitevolume}. 
Similarly, Cheng and Zhou \cite{chengzhouoneend} showed that an $\frac{n-2}{n}$-stable complete minimal hypersurface in $\mathbb{R}^{n+1}$ either has one end or is a catenoid (two ends) if the norm of the second fundamental form in a geodesic ball with radius $R$, i.e., $B_R^{n+1}(p)$, is $o(\ln R)$ when $n=3$ and $o(\sqrt{R})$ when $n=4$ as $R\rightarrow \infty$.
In our case, the second fundamental form is uniformly bounded and the $\delta$-stability implies $\frac{n-2}{n}$-stability.
Thus we can assume $\delta$-stable minimal hypersurface in $\mathbb{R}^3$ or $\mathbb{R}^4$ has only one end. In conclusion, to prove the curvature estimate, we shall prove the claim:
   \vskip.2cm
   \textit{Let  $n=3,4$ and $\delta>\frac{n(n-2)}{4(n-1)}$. Then a complete, two-sided, simply-connected, $\delta$-stable, minimal hypersurface $M$ with finite ends in $\mathbb{R}^{n+1}$ must be flat.}
   \vskip.2cm

The rest of the proof focuses on obtaining the Euclidean volume growth and intrinsic volume growth of $\delta$-stable minimal hypersurfaces. For this purpose, we prove the following Proposition \ref{thm:deriveAreaGrowth}. Combining it with Corollary \ref{thm:mainAreaGrowth}  and Proposition \ref{ssy-delta-stable} proves the above claim, thus completes the proof of \ref{thm:itSSYdelta} of Theorem \ref{delta-stable}.

\begin{proposition}
	Let  $n=3,4$ and $\delta>\frac{n-2}{n}$. Suppose $M^n$ is a simply-connected,  two-sided, $\delta$-stable,  immersed minimal hypersurface, $0\in M$, with finite ends in $\mathbb{R}^{n+1}$.
	Then there exists a constant $C=C(n,\delta)$ such that following holds:
 \begin{enumerate}[\normalfont(a)]
        \item \label{prop:itEuclidean}
        If $M$ is a properly immersion, then we have 
        \[
		\mathcal{H}^n(\tilde{M}_R)\le C R^n
	\]
	for any $R>0$ where $\tilde{M}_R$ is the connected component of $M\cap B^{n+1}_R(0)$ containing the origin.
        \item \label{prop:itGeodesic}
        We have
        \[
            \mathcal{H}^n(B^M_R(0))\le CR^n
        \]
        for any $R>0$.
    \end{enumerate}
	
	\label{thm:deriveAreaGrowth}
\end{proposition}
To prove Proposition \ref{thm:deriveAreaGrowth}, we refine the method of Chodosh-Li \cite{chodoshliR4anisotropic} and Chodosh-Li-Minter-Stryker \cite{chodoshliR5}.

\begin{remark}
    We believe that Theorem \ref{delta-stable} holds for $n=5$ as well.
    The main challenge lies in Proposition \ref{thm:deriveAreaGrowth}. Recently, Mazet \cite{mazet} proved the stable Bernstein Theorem for stable minimal hypersurface in $\mathbb{R}^6$. Using the same technique, we find that when $\delta>\delta_0$ for a real number $\delta_0\approx 0.811\cdots$, we can obtain Euclidean volume growth and establish the corresponding $\delta$-stable Bernstein theorem for $n=5$.
 However, $\delta_0$ is far from the optimal value $\frac{3}{5}$, so we will not delve into further details.
 Not that such $\delta_0$ can imply Theorem \ref{delta-stable} \ref{thm:itSSYdelta} holds for $n=5$.
\end{remark}

\begin{remark}
    After we finished this work, we were informed by Zhenxiao Xie that he obtained \ref{thm:itSSYdelta} of Theorem \ref{delta-stable}
    when $n=3$ for  $\delta>\frac{13}{32}(>\frac{3}{8})$ using a different method.
\end{remark}

\begin{remark}
    One open question is whether extra assumptions in Theorem \ref{delta-stable} are redundant.
\end{remark}

\subsection{Critical case} Lastly, we consider the critical case, where $\delta=\frac{n-2}{n}$.
Previously, it was mentioned that $M$ is either flat or a catenoid under certain curvature growth condition \eqref{tamzhoutotalcurvaturegrowthcondition} if $M$ is $\frac{n-2}{n}$-stable (cf. \cite{tamzhou,chengzhouoneend}).
We conjecture that if $M$ is $\frac{n-2}{n}$-stable, then $M$ is either flat or a catenoid for $3\le n\le 5$.
During the proof of Theorem \ref{thm_regularity_and_compactness_theorem}, we found another partial solution to this conjecture for $n = 3$.

\begin{theorem}
Suppose $M$ is a complete two-sided $\frac{1}{3}$-stable minimal hypersurface in $\mathbb{R}^{4}$ with the following intrinsic volume growth estimate
\begin{equation}
    \mathcal{H}^n(B_R^M(p))\le \Lambda R^n
\end{equation}
for every $R>0$ and some $p\in M$.
If $M$ is contained in the region bounded by two parallel hyperplanes, then $M$ is either a catenoid or a hyperplane.
	\label{thm:catenoidMain}
\end{theorem}

\vskip.2cm
The rest of this paper is organized as follows.
In Section \ref{sec:notation}, we introduce the notation used throughout the paper and present some fundamental concepts regarding varifolds.
In Sections \ref{sec:curvature_estimate_with_area_conditions} and \ref{sec:tangent_cone_analysis}, we prove Theorem \ref{thm_regularity_and_compactness_theorem}.
In Section \ref{sec:curvature_estimate_with_area_conditions}, we establish the $L^\infty$ estimate for the second fundamental form (Theorem \ref{thm:epsilonRegularity}) given that the $L^{2\alpha}$ norm of the second fundamental form is small.
In Section \ref{sec:tangent_cone_analysis}, we classify $\delta$-stable minimal cones (Proposition \ref{thm:deltaCone}) and use the standard cone analysis arguments to complete the proof of Theorem \ref{thm_regularity_and_compactness_theorem}.
In Section \ref{sec:area_growth_estimate}, we consider the conformal metric for $\delta$-stable minimal hypersurfaces in $\mathbb{R}^4$ and $\mathbb{R}^5$ and further divide it into two subsections to construct the $\mu$-bubbles under conformal metric for $n=3$ and $n=4$ separately.
Then, we obtain the desired volume growth estimate for $\delta$-stable minimal hypersurfaces (Proposition \ref{thm:deriveAreaGrowth}).
Finally, we prove Theorem \ref{thm:catenoidMain} in Section \ref{sec:critical_case}.

\subsection{Acknowledgements}
We would like to thank Costante Bellettini and Chao Li for their valuable comments on our draft and for clarifying aspects of their papers. The authors would like to thank Zhenxiao Xie for his interest. The first author is supported by Beijing Jiaotong University Science grant No. YA24XKRC00010. The second author is supported by NSFC Grant No. 11831005.
    
\section{Notation and definitions}%
\label{sec:notation}

We first introduce some notations and concepts.

We denote $x=(x_1,x_2,\cdots ,x_{n+1})$ as the point in $\mathbb{R}^{n+1}$ and, unless otherwise stated,  use $\left\{ e_1,e_2,\cdots,e_{n+1} \right\}$ to represent the standard orthonormal basis in $\mathbb{R}^{n+1}$. Let $B^N_r(x)$ represent the open ball in $\mathbb{R}^{N}$ with radius $r$ and center $x$. For any $y \in \mathbb{R}^{n+1}$ and $r>0$, $\eta_{y,r}:\mathbb{R}^{n+1}\rightarrow \mathbb{R}^{n+1}$ denotes the map $\eta_{y,r}(x)=\frac{1}{r}(x-y)$.

We say $M^n$ is a \textit{smooth immersed hypersurface} in an open subset $U\subset \mathbb{R}^{n+1}$ if $M$ is an $n$-dimensional manifold and there exists a smooth immersion $\iota:M\rightarrow U$.
In general, we will identify $M$ with its image $\iota(M)$ if there is no confusion.

Given a smooth immersed hypersurface $M$ in $U$ (or just a subset of $U$), for any $x \in \bar{M}\cap U$, we say $x$ is a \textit{regular point} of $M$ if there exists a number $r>0$ such that $B^{n+1}_r(x)\subset U$ and $\bar{M}\cap B^{n+1}_r(x)$ can be written as a union of finitely many smooth, compact, connected, embedded hypersurfaces $\Sigma_i$ in $B^{n+1}_r(x)$ such that $\bar{\Sigma}_i\cap B^{n+1}_r(x)=\Sigma_i$.
In general, we redefine $M$ such that each point in $M$ is a regular point and every regular point of $M$ lies in $M$.

The (interior) \textit{singular set} of $M$ is defined by
\[
	\mathrm{sing}M=(\bar{M}\backslash M)\cap U.
\]

\begin{remark}
    If $M$ is a complete minimal hypersurface immersed in $\mathbb{R}^{n+1}$ satisfying the Euclidean volume growth condition \ref{eq:areaGrowthCondition}, then $M$ is a properly immersed hypersurface (e.g., see the proof of Lemma 4 in \cite{Tysk1989finitenessIndex}).
    Hence, $\mathrm{sing} M=\emptyset$.
\end{remark}

Denote $\nu$ as the unit normal vector field of $M$. The second fundamental form of $M$ at $x \in M$ is defined by
\[
	A(X,Y)= -D_XY\cdot \nu
\]
for any $X,Y\in T_xM$, where $D$ represents the standard covariant derivative in $\mathbb{R}^{n+1}$ and $X\cdot Y$ represents the standard inner product in $\mathbb{R}^{n+1}$.

We say $M$ is \textit{minimal} if it is a critical point of the area functional for any compactly supported variation $\iota_t:M\rightarrow \mathbb{R}^{n+1}$ of $\iota$.
This is equivalent to the mean curvature $H$ of $M$ being zero.

Given $\delta>0$, we say a smooth immersed hypersurface $M$ is $\delta$-stable if, for any smooth function $\varphi$ with compact support defined on $M$ (not on image), we have
\[
	\delta \int_{ M} |A|^2\varphi^2\le \int_{ M} |\nabla \varphi|^2.
\]

Note that $1$-stable is equivalent to the usual stable condition.
That is, $M$ is 1-stable if and only if
\[
	\left.\frac{d^2}{dt^2}\right|_{t=0}\mathcal{H}^n(\iota_t(M\cap K))\ge 0
\]
for any smooth variation $\iota_t:M\rightarrow \mathbb{R}^{n+1}$ of $\iota$ supported in a compact subset $K\subset \subset M$.

We also present some basic definitions and facts about varifolds.
Readers can refer to \cite{Allard1972,simon1983lectures} for more details. 

An \textit{$n$-varifold} $V$ in $\mathbb{R}^{n+1}$ is a Radon measure on $\mathbb{R}^{n+1}\times G(n+1,n)$ where $G(n+1,n)$ is the set of all $n$-dimensional subspaces in $\mathbb{R}^{n+1}$ (the Grassmannian).
The weight measure $\|V\|$ of an $n$-varifold $V$ is defined by
\[
	\|V\|(A)=V(A \times G(n+1,n))
\]
for any Borel subset $A \subset \mathbb{R}^{n+1}$.
Hence, $\|V\|$ is a Radon measure on $\mathbb{R}^{n+1}$.
The support of $\|V\|$, denoted by $\mathrm{spt}\|V\|$, is defined by
\[
	\mathrm{spt}\|V\|=\left\{ x \in \mathbb{R}^{n+1}: \|V\|(B^{n+1}_r(x))>0 \text{ for any }r>0 \right\}.
\]
The ($n$-dimensional) density of $\|V\|$ at $x$ is defined by
\[
	\Theta(\|V\|,x)=\lim_{r\rightarrow 0^+}\frac{\|V\|(B^{n+1}_r(x))}{\omega_nr^n},
\]
if the limit exists, where $\omega_n$ is the volume of the unit ball in $\mathbb{R}^n$.

For any proper smooth immersed hypersurface $M$, we denote $|M|$ as the $n$-varifold associated with $M$.

We say $V$ is \textit{stationary} in $U$ if, for any $C^1$ variation $F_t:U\rightarrow U$ supported on $K\subset \subset U$ with $F_0$ being the identity map, we have
\[
	\left.\frac{d}{dt}\right|_{t=0}\|(F_t)_{\#}V\|(K)=0.
\]
where $(F_t)_{\#}V$ is the pushforward of $V$ by $F_t$.
The tangent cone of an $n$-varifold $V$ at $x$, denoted as $\mathrm{VarTan}(V,x)$, is defined as follows:
\[
\mathrm{VarTan}(V,x):=
\{V':V'=\lim_{i\to \infty}(\eta_{x,\rho_i})_\#V\text{ for some }\rho_i\to 0^+
\}.
\]

We say a sequence of immersed hypersurfaces $M_k$ converges to an $n$-varifold $V$ in the sense of varifolds if $|M_k|$ converges to $V$ in the sense of Radon measures. At last, for any $n$-varifold defined on $U$, $\mathrm{sing}\|V\|$ denotes the singular set of $\mathrm{spt}\|V\|$ in $U$.
\section{$L^\infty$ estimate of the second fundamental form}%
\label{sec:curvature_estimate_with_area_conditions}


One of the key ingredients in the proof of Theorem \ref{thm_regularity_and_compactness_theorem} is the following $\varepsilon$-regularity theorem.
\begin{theorem}
	[$\varepsilon$-regularity for $|A|$ under $\delta$-stable condition]
	Let $n\ge 3$ and $\delta>\frac{n-2}{n}$. 
	Suppose $M^n$ is a two-sided $\delta$-stable minimal hypersurface immersed in $B^{n+1}_4(0)$ and the singular set of $M$ satisfies $\bar{n}:=\mathrm{dim}(\mathrm{sing}M)< n-2-\frac{2(n-2)}{n}$.
	Additionally, assume $\mathcal{H}^n(M\cap B^{n+1}_4(0))\le \Lambda$ for some $\Lambda \in (0,+\infty)$.
	Then, for any $\alpha  \in (\frac{n-2}{n},\min \left\{ \delta, \frac{n-\bar{n}-2}{2},1 \right\})$, 
	there exists $\varepsilon=\varepsilon(n,\bar{n},\delta,\alpha,\Lambda) \in (0,1)$ such that if
	\[
		\int_{ B^{n+1}_2(0)\cap M} |A|^{2\alpha}\le \varepsilon,
	\]
	then
	\[
		\sup_{B^{n+1}_{\frac{1}{2}}(0)\cap M}|A|^{2\alpha}\le C \int_{ B^{n+1}_2(0)\cap M} |A|^{2\alpha}
	\]
	for some constant $C=C(n,\bar{n},\delta,\Lambda,\alpha)$.
	\label{thm:epsilonRegularity}
\end{theorem}

For simplicity, we write $u=|A|^\alpha$.
We need the following lemma (weak (intrinsic) Caccipoppli inequality).

\begin{lemma}	
	For any $\alpha \in (\frac{n-2}{n},\min \left\{ \delta, \frac{n-\bar{n}-2}{2} ,1\right\})$, and any locally Lipschitz function $\phi$ supported in $B^{n+1}_3(0)$, we have
	\begin{align}
		\int_{M\cap \{u>k\}} \left( 1-\frac{k}{u} \right){}&|\nabla u|^2\phi^2\le
		C\int_{M\cap \{u>k\}} (u-k)^2|\nabla \phi|^2\nonumber\\+{}&Ck^2\int_{ M\cap \{u>k\}} \left( (u -k)^{\frac{2}{\alpha}}+k^{\frac{2}{\alpha}} \right) \phi^2,
		\label{eq:lemWeakCacci}
	\end{align}
	where the constant $C=C(n,\bar{n}, \delta,\Lambda,\alpha)$.
	\label{lem:weakPoincare}
\end{lemma}

\begin{proof}
    At first, we show that \eqref{eq:lemWeakCacci} holds for $\phi$ which is a bounded locally Lipschitz function with compact support in $B^{n+1}_4(0)$, vanishing in a neighborhood of $\mathrm{sing}M \cap B^{n+1}_4(0)$ and any $\alpha \in (\frac{n-2}{n},\min\{\delta,1\})$.
    
For such $\phi$ and $\alpha$, we choose $\varphi = (|A|^\alpha-k)^+ \phi$ for $k\ge 0$.
Note that we can verify $((|A|^\alpha-k)^+)^2 \in C^1(M)\cap W^{2,\infty}_{\mathrm{loc}}(M)$.
Hence, we can insert such $\varphi$ into the $\delta$-stability inequality \eqref{eq:deltaStability}.

We observe that
\begin{align*}
\frac{1}{2}\Delta(|A|^\alpha-k)^2={}&\alpha\left( 1-\frac{k}{|A|^{\alpha}} \right)|A|^{2\alpha-2}|A|\Delta|A|\\
+&\alpha \left( \left( 1-\frac{k}{|A|^\alpha} \right)(\alpha-1)+\alpha \right)|A|^{2\alpha-2}|\nabla|A||^2.
\end{align*}
Using Simon's inequality
\[
	|A|\Delta |A|\ge\frac{2}{n}|\nabla|A||^2-|A|^4,
\]
then we obtain
\begin{align*}
	&\int_{M} |\nabla \varphi|^2\\={} & \int_{M_{>k}} \alpha^2|A|^{2\alpha-2}|\nabla|A||^2\phi^2+((|A|^\alpha-k)^2|\nabla \phi|^2+\frac{1}{2}\left< \nabla (|A|^\alpha-k)^2, \nabla \phi^2 \right> \\
	={} & \int_{M_{>k}} \alpha^2|A|^{2\alpha-2}|\nabla|A||^2\phi^2+((|A|^\alpha-k)^2|\nabla \phi|^2-\frac{1}{2}\phi^2\Delta(|A|^\alpha-k)^2\\
	 ={}& \int_{M_{>k}} ((|A|^\alpha-k)^2)|\nabla\phi|^2
	 -\int_{M_{>k}} \alpha \left( 1-\frac{k}{|A|^\alpha} \right)|A|^{2\alpha-2}|A|\Delta|A|\phi^2\\
  &+\alpha(\alpha-1)\left( 1-\frac{k}{|A|^\alpha} \right)|A|^{2\alpha-2}|\nabla|A||^2\phi^2\\
	 \le{}& \int_{M_{>k}} ((|A|^\alpha-k)^2)|\nabla\phi|^2\\
		  &-\int_{M_{>k}} \alpha\left( \frac{2}{n}+\alpha-1 \right) \left( 1-\frac{k}{|A|^\alpha} \right)|A|^{2\alpha-2}|\nabla|A||^2\phi^2
		  +\int_{M_{>k}} \alpha |A|^{2\alpha+2}\phi^2
\end{align*}
where $M_{>k}$ denotes $M\cap \{|A|^\alpha>k\}$.
On the other hand, by the $\delta$-stability, we have
\[
	\int_{M} |\nabla\varphi|^2\geq \delta \int_{M} |A|^2\varphi^2=\delta \int_{M_{>k}} |A|^{2}(|A|^{\alpha}-k)^2\phi^2. 
\]
Now, we write $\delta_1:=\alpha - \frac{n-2}{n}>0$, $\delta_2:=\delta - \alpha>0$.
Then, from the stability inequality it follows
\begin{align*}
	& \delta_1\int_{M_{>k}} \alpha \left( 1-\frac{k}{|A|^\alpha} \right)|A|^{2\alpha-2}|\nabla|A||^2\phi^2\\
	\le{}& \int_{M_{>k}} ((|A|^\alpha-k)^2)|\nabla \phi|^2
	+\int_{M_{>k}} \alpha |A|^{2\alpha}|A|^2\phi^2\\
		 &-\delta \int_{M_{>k}} |A|^{2}(|A|^{\alpha}-k)^2\phi^2.
\end{align*}
Now, let $u=|A|^\alpha$.
Then,
\begin{equation}
	\frac{\delta_1}{\alpha} \int_{M_{>k}} \left( 1-\frac{k}{u} \right)|\nabla u|^2\phi^2\le \int_{M_{>k}} (u-k)^2|\nabla \phi|^2+\int_{M_{>k}}u^{\frac{2}{\alpha}}\left( \alpha u^2-\delta (u-k)^2 \right) \phi^2.
	\label{eq:pfWeakPoincareExtraTerm}
\end{equation}
Now, we need to estimate $u^{\frac{2}{\alpha}}\left( \alpha u^2-\delta(u-k)^2 \right)$.
\begin{align*}
	\alpha u^2-\delta(u -k)^2={} & -\delta_2 (u -k)^2+2\alpha k(u-k)+\alpha k^2
	\le \frac{\alpha^2k^2}{\delta_2}+\alpha k^2,
\end{align*}
where we have used Cauchy-Schwarz inequality.
By the trivial inequality $(x+y)^a\le 2^a(x^a+y^a)$ for $x,y\ge 0$, $a>0$, we have
\[
	u^{\frac{2}{\alpha}}(\alpha u^2-\delta(u -k)^2)\le 2^{\frac{2}{\alpha}}\left( (u -k)^{\frac{2}{\alpha}}+k^{\frac{2}{\alpha}} \right)k^2\left( \alpha + \frac{\alpha^2}{\delta_2} \right).
\]
Substituting this into inequality \eqref{eq:pfWeakPoincareExtraTerm}, we can get the desired inequality \eqref{eq:lemWeakCacci}.

To complete the proof, we need to show that \eqref{eq:lemWeakCacci} holds for any bounded locally Lipschitz function $\phi$ supported in $B_3(0)$ by an approximation argument if we assume $\alpha \in (\frac{n-2}{n},\min \left\{ \delta,\frac{n-\bar{n}-2}{2},1 \right\})$.
Note that $\phi$ may be non-zero on the singular set of $M$.
Such an argument has been used by Schoen \cite{Schoen1977thesis} and Schoen-Simon \cite{SchoenSimon1981Regularity}.

Let us first get a preliminary estimate on $|A|$.
\begin{lemma}
	\label{lem_boundedLpNorm}
	If $\mathcal{H}^{n-2}(\mathrm{sing}(M)\cap B^{n+1}_4(0))=0$, then we have $|A| \in L^2(B^{n+1}_{\frac{7}{2}}(0)\cap M)$ and 
	\[
		\int_{M\cap B^{n+1}_\rho(x) } |A|^2\le C\rho^{n-2},
	\]
	for any $x \in B^{n+1}_{\frac{7}{2}}(0)$ and $\rho \in (0,\frac{1}{4})$, where $C=C(\delta,\Lambda)$.
\end{lemma}

\begin{proof}
	For each $\varepsilon>0$, we choose balls $\left\{ B^{n+1}_{r_i}(x_i) \right\}_{i=1}^N$ such that $\mathrm{sing}(M)\cap B^{n+1}_4(0)\subset \bigcup_{i=1}^N B^{n+1}_{r_i}(x_i)$ and $\sum_{i=1}^N r_i^{n-2}\le \varepsilon$.
	We choose $\zeta_i$ to be a non-negative $C^1$ function such that $\zeta_i$ is supported outside of $B^{n+1}_{r_i}(x_i)$, $\zeta_i=1$ outside of $B^{n+1}_{2r_{i}}(x_i)$, and $|\nabla \zeta_i|\le \frac{2}{r_i}$.
	Then, we define $\zeta_\varepsilon=\min_{1\le i\le N}\zeta_i$.
	We insert $\zeta_\varepsilon \phi$ into the $\delta$-stability inequality where $\phi$ is a non-negative locally Lipschitz function with compact support in $B^{n+1}_4(0)$.
	Then,
	\[
		\int_{M\cap  B^{n+1}_4(0)} |A|^2\phi^2\zeta_\varepsilon^2\le \frac{2}{\delta}\int_{M\cap  B^{n+1}_4(0)} \left|\nabla \phi\right|^2\zeta_\varepsilon^2+\frac{2}{\delta}\int_{M\cap  B^{n+1}_4(0)}|\nabla \zeta_\varepsilon|^2\phi^2
	\]
	by Cauchy-Schwarz inequality.
	Note that
	\begin{align*}
		\int_{M\cap  B^{n+1}_4(0)} |\nabla \zeta_\varepsilon|^2\phi^2\le{} & \|\phi\|_{L^\infty(B^{n+1}_4(0))}^2\sum_{i=1}^{N}\int_{M\cap  B^{n+1}_{2r_i}(x_i)} |\nabla \zeta_i|^2\\
		 \le{}& C\|\phi\|_{L^\infty(B^{n+1}_4(0))}^2\sum_{i=1}^{N}r_i^{n-2}\le C\|\phi\|_{L^\infty(B^{n+1}_4(0))}^2\varepsilon
	\end{align*}
	which converges to $0$ as $\varepsilon\rightarrow 0^+$.
	Then, we have
	\[
		\int_{M} |A|^2\phi^2\le \frac{2}{\delta}\int_{M}|\nabla \phi|^2. 
	\]
	In particular, it implies $|A| \in L^2(B^{n+1}_{\frac{7}{2}}(0))$ if we choose $\phi\equiv 1$ on $B^{n+1}_{\frac{7}{2}}(0)$.
	Now, we choose $\phi$ supported on $B^{n+1}_{2\rho}(x)$, and equal to $1$ on $B^{n+1}_\rho(x)$, with $|\nabla \phi|\le \frac{2}{\rho}$.
	Together with the monotonicity formula, we have
	\[
		\int_{M\cap  B^{n+1}_\rho(x)} |A|^2\le \frac{2}{\delta}\int_{M\cap  B^{n+1}_{2\rho}(x)\backslash B^{n+1}_\rho(x)} \frac{1}{\rho^2}\le C\rho^{n-2},
	\]
	for some $C=C(\delta,\Lambda)$.
\end{proof}

The remaining part is similar to the proof of the above lemma.
Based on the assumption of $\bar{n}$ and $\alpha$, we know $\mathcal{H}^{n-2-2\alpha}(\mathrm{sing}M)=0$.
Therefore, for any $\varepsilon>0$, there exist $B^{n+1}_{r_1}(x_1),B^{n+1}_{r_2}(x_2),\cdots , B^{n+1}_{r_N}(x_N)$ with $x_i\in B^{n+1}_{\frac{7}{2}}(0)$ and $0<r_i<\frac{1}{4}$ for each $1\le i\le N$, such that
\begin{equation}
	\mathrm{sing}M\cap B^{n+1}_3(0)\subset \bigcup_{i=1}^N B^{n+1}_{r_i}(x_i),\quad \text{ and }\quad \sum_{i =1}^{N}r_i^{n-2-2\alpha}\le \varepsilon.
	\label{eq:pfSingCover}
\end{equation}
We choose $\zeta_i$ and $\zeta_\varepsilon$ as in the proof of Lemma \ref{lem_boundedLpNorm}.
For any $\phi$ which is locally Lipschitz with compact support in $B^{n+1}_3(0)$, we know $\zeta_\varepsilon \phi$ vanishes near $\mathrm{sing}M$, allowing us to use \eqref{eq:lemWeakCacci} with $\zeta_\varepsilon \phi$ in place of $\phi$.
Thus, we have
\begin{align}
	{} & \int_{M_{>k}} \left( 1-\frac{k}{u} \right)|\nabla u|^2\zeta_\varepsilon^2\phi^2\nonumber \\
	\le 
	{}& C\int_{M_{>k}} (u-k)^2\zeta_\varepsilon^2|\nabla \phi|^2+Ck^2\int_{M_{>k}} \left( (u -k)^{\frac{2}{\alpha}}+k^{\frac{2}{\alpha}} \right) \zeta_\varepsilon^2\phi^2\nonumber \\
	{}& + C \int_{M_{>k}} (u-k)^2|\nabla\zeta_\varepsilon|^2\phi^2,
	\label{eq:pfWeakPoincareApprox}
\end{align}
by the Cauchy-Schwarz inequality.

For the first two terms on the right-hand side of \eqref{eq:pfWeakPoincareApprox}, since $|A|^{2\alpha}$ and $|A|^2$ are integrable in $B^{n+1}_3(0)\cap M$ by Lemma \ref{lem_boundedLpNorm},
we can let $\varepsilon\rightarrow 0^+$, leading to
\[
	C \int_{M_{>k}} (u-k)^2|\nabla \phi|^2+Ck^2\int_{M_{>k}} \left( (u -k)^{\frac{2}{\alpha}}+k^{\frac{2}{\alpha}} \right) \phi^2.
\]
Then, we need to show
\[
	\lim_{\varepsilon\rightarrow 0^+} \int_{M_{>k}} (u-k)^2|\nabla\zeta_\varepsilon|^2\phi^2=0.
\]
Applying Lemma \ref{lem_boundedLpNorm}, \eqref{eq:pfSingCover}, and H\"older inequality, we obtain
\begin{align*}
	{} & \int_{M_{>k}} (u-k)^2|\nabla \zeta_\varepsilon|^2\phi^2\\
	\le 
	{}& 
	\sum_{i=1}^{N}\|\phi\|_{L^\infty(B^{n+1}_3(0))}^2\int_{M\cap   B^{n+1}_{r_i}(x_i)} |A|^{2\alpha}|\nabla \zeta_i|^2\\
	\le{}& \|\phi\|_{L^\infty(B^{n+1}_3(0))}^2\sum_{i=1}^{N}\left( \int_{M\cap  B^{n+1}_{r_i}(x_i)} |A|^2 \right)^{\alpha}\left( \int_{M\cap  B^{n+1}_{r_i}(x_i)} |\nabla \zeta_i|^{\frac{2}{1-\alpha}} \right)^{1-\alpha}\\
	\le 
	{}& 
	C\|\phi\|_{L^\infty(B^{n+1}_3(0))}^2\sum_{i=1}^{N}r_i^{\alpha(n-2)}r_i^{\left(n-\frac{2}{1-\alpha}\right)(1-\alpha)}=C\|\phi\|_{L^\infty(B^{n+1}_3(0))}^2\sum_{i=1}^{N}r_i^{n-2-2\alpha}\\
	\le 
	{}& C\|\phi\|_{L^\infty(B^{n+1}_3(0))}^2\varepsilon.
\end{align*}
Hence, we can take $\varepsilon\rightarrow 0^+$ to get \eqref{eq:lemWeakCacci} holds for any bounded locally Lipschitz function $\phi$ supported in $B^{n+1}_3(0)$.
\end{proof}
\begin{remark}
	From the above proof, we can easily see that $|\nabla u|^2$ is integrable in $B^{n+1}_3(0)\cap M$ and hence $u\in W^{1,2}(B^{n+1}_3(0)\cap M)$.
\end{remark}
We now prove the Theorem \ref{thm:epsilonRegularity}.
\begin{proof}
	[Proof of Theorem \ref{thm:epsilonRegularity}]
Consider
\[
	k_l=d\left( 1 - \frac{1}{2^{l-1}} \right),\quad \text{and}\quad R_l=\frac{1}{2}+\frac{1}{2^l},
\]
for $0<d\le 1$. $k_l$ increases to $d$ and $R_l$ decreases to $1/2$ as $l\rightarrow \infty.$
For simplicity, we write $\Omega_l = M\cap \left\{ u>l \right\}\cap B^{n+1}_{R_l}(0)$.

Applying the previous lemma and noting that
\[
	1-\frac{k_l}{u}\ge 1-\frac{k_l}{k_{l+1}}\ge \frac{1}{2^l},
\]
for any $u>k_{l+1}$, we have
\begin{align*}
	\frac{1}{2^l}\int_{ M_{>k_{l+1}}} |\nabla u|^2 \phi^2\le{}&
	C\left[ \int_{ M_{>k_{l}}} (u-k_{l})^2|\nabla \phi|^2\right.\\
&\left.+d^2\int_{ M_{>k_{l}}}  (u -k_{l})^{\frac{2}{\alpha}}\phi^2+
	d^{2+\frac{2}{\alpha}}\int_{ M_{>k_l}}\phi^2 \right] .
\end{align*}
Use the fact
\[
	|\nabla((u-k_{l+1})\phi)|^2\le 2|\nabla u|^2\phi^2+2(u-k_{l+1})^2|\nabla \phi|^2,
\]
we obtain
\begin{align*}
	{} & \int_{ M_{>k_{l+1}}} |\nabla ((u-k_{l+1})\phi)|^2\\
	\le 
	{}& 2^l C\left[ \int_{ M_{>k_{l}}} (u-k_{l})^2|\nabla \phi|^2+d^2\int_{ M_{>k_{l}}}  (u -k_{l})^{\frac{2}{\alpha}}\phi^2+
	d^{2+\frac{2}{\alpha}}\int_{ M_{>k_l}}\phi^2 \right] .
\end{align*}
Now, we choose $\phi$ supported in $B^{n+1}_{R_l}(0)$ and $\phi=1$ on $B^{n+1}_{R_{l+1}}(0)$, and $|\nabla \phi|\le 2^{l+2}$ with $0\le \phi \le 1$.
Together with Michael-Simon's inequality
\[
		\left( \int_{M} |\varphi|^{\frac{2n}{n-2}} \right)^{\frac{n-2}{n}}\le C \int_{M} \left|\nabla \varphi\right|^2,
	\]
for a constant $C$ only depending on $n$.
Then, we have
\begin{align}
	{} & \left( \int_{ \Omega_{l+1}} (u-k_{l+1})^{\frac{2n}{n-2}} \right)^{\frac{n-2}{n}}\nonumber \\
	\le 
	{}& C^l\left[ \int_{ \Omega_l} (u-k_{l})^2+d^2\int_{ \Omega_l}  (u -k_{l})^{\frac{2}{\alpha}}+
	d^{2+\frac{2}{\alpha}}\mathcal{L}^n(\Omega_l) \right].
	\label{eq:pfIterEst}
\end{align}


Using the fact that when $u\ge k_l$, we know $u-k_{l-1}\ge \frac{d}{2^{l-1}}$. Hence, for any $0\le \beta\le \frac{2n}{n-2}$,
\begin{align*}
	\int_{ \Omega_l} (u-k_l)^\beta\le{} & \int_{ \Omega_l}(u-k_l)^\beta \left( \frac{2^{l-1}}{d} \right)^{\frac{2n}{n-2}-\beta}(u-k_{l-1})^{\frac{2n}{n-2}-\beta}\\
	 \le{}& \frac{C^l}{d^{\frac{2n}{n-2}-\beta}}\int_{ \Omega_{l-1}} (u-k_{l-1})^{\frac{2n}{n-2}},
\end{align*}
where constant $C=C(n)$.
Note that since $\frac{2}{\alpha}< \frac{2n}{n-2}$, we can use the above inequality with $\beta=0,2$, and $\frac{2}{\alpha}$ in \eqref{eq:pfIterEst} to obtain
\begin{equation}
	S_{l+1}^{\frac{n-2}{n}}\le C^l\left( \frac{1}{d^{\frac{2n}{n-2}-2}}+\frac{1}{d^{\frac{2n}{n-2}-\frac{2}{\alpha}-2}}+\frac{1}{d^{\frac{2n}{n-2}-2-\frac{2}{\alpha}}} \right)S_{l-1},
	\label{eq:pfIterRaw}
\end{equation}
where
\[
	S_l:= \int_{\Omega_l}  (u-k_{l})^{\frac{2n}{n-2}}.
\]
Using $d\le 1$, \eqref{eq:pfIterRaw} implies
\begin{equation}
		\frac{S_{l+1}}{d^{\frac{2n}{n-2}}}\le
	C^l \left( \frac{S_{l-1}}{d^{\frac{2n}{n-2}}} \right)^{\frac{n}{n-2}},
	\label{eq:pfIter}
\end{equation}
for some $C=C(n,\bar{n},\delta,\Lambda,\alpha)$.
By iterating \eqref{eq:pfIter}, we have
\[
	\frac{S_{2l+1}}{d^{\frac{2n}{n-2}}}\le
	C^{2+\frac{4(n-2)}{n}+\cdots + 2l (\frac{n-2}{n})^{l-1}}\left( C^2\frac{S_1}{d^{\frac{2n}{n-2}}} \right)^{\left( \frac{n}{n-2} \right)^l}\le C^{\frac{n^2}{2}}\left( C^2\frac{S_1}{d^{\frac{2n}{n-2}}} \right)^{\left( \frac{n}{n-2} \right)^l}.
\]
Hence,
if we require
\[
	S_1\le (\varepsilon' d)^{\frac{2n}{n-2}},
\]
for some positive $\varepsilon'$ only depending on $n$, $\bar{n}$, $\delta$, $\Lambda$, and $\alpha$,
then we have $\lim_{l\rightarrow \infty} S_{2l+1}=0$.
This implies
\[
	\left|A\right|^\alpha\le d \text{ on }B^{n+1}_{\frac{1}{2}}(0).
\]

Lastly, we need to ensure $S_1\le (\varepsilon' d)^{\frac{2n}{n-2}}$.

Using Lemma \ref{lem:weakPoincare} with $k=0$ and a suitable test function, we obtain
\[
	\int_{M\cap  B^{n+1}_{\frac{3}{2}}(0)} |\nabla u|^2 \le
	C\int_{M\cap  B^{n+1}_2(0)}u^2
\]
for some $C=C(n,\bar{n},\delta,\Lambda,\alpha)$.
Thus, by Michael-Simon's inequality, we have
\begin{equation}
	S_1^{\frac{n-2}{n}}\le C \int_{M\cap  B^{n+1}_1(0)} \left|\nabla (u\varphi)\right|^2 \le C \int_{M\cap  B^{n+1}_2(0)} u^2=C\int_{M\cap  B^{n+1}_2(0)} |A|^{2\alpha}.
	\label{eq:pfMSLastEst}
\end{equation}
for $\varphi$ supported on $B^{n+1}_{\frac{3}{2}}(0)$, equal to $1$ on $B^{n+1}_1(0)$, and $|\nabla \varphi|\le 4$, where $C=C(n,\bar{n},\delta,\Lambda,\alpha)$.
Now, we choose $\varepsilon=\frac{(\varepsilon')^2}{C}$ where $C$ is the constant in \eqref{eq:pfMSLastEst} and $d =\sqrt{\frac{1}{\varepsilon}\int_{M\cap  B^{n+1}_2(0)\cap M} |A|^{2\alpha}} \in (0,1]$ by the assumption.
Consequently, $S_1^{\frac{n-2}{n}}\le (\varepsilon' d)^2$ holds by \eqref{eq:pfMSLastEst}.
For such a choice of $d$, we know $\lim_{l\rightarrow \infty} S_l=0$, which implies
\[
	\sup_{B^{n+1}_{\frac{1}{2}}(0)\cap M}|A|^{2\alpha}\le d^2 =C \int_{M\cap  B^{n+1}_2(0)\cap M} |A|^{2\alpha}
\]
for some $C=C(n,\bar{n},\delta,\Lambda,\alpha)$.
\end{proof}

\section{Tangent Cone Analysis}
\label{sec:tangent_cone_analysis}

\subsection{Classification of $\delta$-stable cone}

In this subsection, we refine the method described in \cite{Li2020bernsteinCone} to demonstrate that a $\delta$-stable minimal cone in $\mathbb{R}^{n+1}$ is flat if $n \geq 2$ and $\delta > \frac{(n-2)^2}{4(n-1)}$.

Consider an $n$-dimensional minimal cone $M$ in $\mathbb{R}^{n+1}$ with the vertex at the origin. Simons' equation (\cite{Schoen1975curvature}, \cite[(2.16)]{Colding2011}) states that,
\begin{equation}
	\frac{1}{2}\Delta|A|^2+|A|^4=\sum_{i,j,k }^{n}A_{ij,k}^2
	\label{eq:SimonsIden}
\end{equation}
where $A$ is the second fundamental form of $M$ and $A_{ij,k}$ is its covariant derivative. We select an orthonormal frame $\left\{ e_1, \ldots, e_n \right\}$ of $M$ such that $e_n = \frac{\partial}{\partial r}$ and $A_{\alpha \beta} = 0$ for $1 \le \alpha < \beta \le n-1$. Since $M$ is a cone, we have $A_{ni} = 0$ and $A_{ij,n} = -\frac{A_{ij}}{r}$ for $1 \le i, j \le n$. Consequently,
\begin{equation}
	\sum_{i,j,k }^{}A_{ij,k}^2=\sum_{\alpha,\beta=1}^{n-1}3A_{\alpha\beta,n}^2+\sum_{\alpha,\beta,\gamma=1 }^{n-1}A_{\alpha\beta,\gamma}^2=\sum_{\alpha=1 }^{n-1}\frac{3A_{\alpha\alpha}^2}{r^2}+\sum_{\alpha,\beta,\gamma=1 }^{n-1}A_{\alpha\beta,\gamma}^2.
	\label{eq:ExpandIICone}
\end{equation}
Recall that
\begin{align}
	|\nabla|A||^2={}& \frac{|\nabla|A|^2|^2}{4|A|^2}= \frac{1}{|A|^2}\sum_{k=1}^n\left( \sum_{i=1}^n h_{ii}h_{ii,k} \right)^2\nonumber \\
	={}&\frac{1}{|A|^2}\sum_{\beta=1}^{n-1}\left( \sum_{\alpha=1 }^{n-1}h_{\alpha\alpha}h_{\alpha\alpha,\beta} \right)^2+\sum_{\alpha =1}^{n-1}\frac{A_{\alpha\alpha}^2}{r^2}\label{eq:pfContorlGrad}\\
	\le{}&\sum_{\alpha,\beta=1 }^{n-1}A_{\alpha\alpha,\beta}^2+\sum_{\alpha =1}^{n-1}\frac{A_{\alpha\alpha}^2}{r^2}\nonumber .
\end{align}
Note that $A_{nn}\equiv 0$, and the minimality condition implies $\sum_{\alpha=1 }^{n-1}A_{\alpha\alpha,i}=0$ for any $1\le i\le n$.
A standard argument (cf. \cite[(2.20)-(2.22)]{Colding2011}) yields
\[
	\left( 1+\frac{2}{n-1} \right)\sum_{\alpha,\beta=1 }^{n-1}A_{\alpha\alpha,\beta}^2\le \sum_{\alpha,\beta,\gamma =1}^{n-1}A_{\alpha\beta,\gamma}^2,
\]
which is equivalent to
\[
	\left( 1+\frac{2}{n-1} \right)\left( |\nabla|A||^2- \frac{|A|^2}{r^2} \right)+\frac{3|A|^2}{r^2}\le \sum_{i,j,k=1 }^{n}A_{ij,k}^2
\]
by $|A|^2=r^2\sum_{\alpha=1 }^{n-1}A_{\alpha\alpha}^2$ and \eqref{eq:ExpandIICone}.

Note that by \eqref{eq:pfContorlGrad}, we have
\[
	|\nabla|A||^2\ge \frac{|A|^2}{r^2}.
\]

Therefore, for any $p \leq 1 + \frac{2}{n-1}$ (potentially negative), we have
\[
	(3-p) \frac{|A|^2}{r^2}+p|\nabla|A||^2\le \sum_{i,j,k =1}^{n}A_{ij,k}^2.
\]
Combining this with Simons' identity \eqref{eq:SimonsIden}, we obtain
\begin{equation}
	|A|\Delta |A|+|A|^4\ge (p-1)|\nabla|A||^2+(3-p)\frac{|A|^2}{r^2}.
	\label{eq:SimonsIneq}
\end{equation}

Now, we are ready to establish the sharp Bernstein theorem for $\delta$-stable minimal cone (Proposition \ref{thm:deltaCone}. The sharpness on $\delta$ due to the Simons-type cone which has been discussed in the introduction.

\begin{proof}
[Proof of Proposition \ref{thm:deltaCone}]
	Note that $n=2$ is trivial.
	So we only focus on the case $n\ge 3$.
        Let $\varepsilon>0$ be any positive real number.
	From \eqref{eq:SimonsIneq}, we have
	\begin{align}
		\frac{1}{2}\Delta(|A|^2+\varepsilon)^{\delta}\ge{}& 
		\delta (|A|^2+\varepsilon)^{\delta-2}\left((2\delta-2+p)|A|^2+p\varepsilon\right)|\nabla|A||^2\nonumber \\
		{}& +\delta(|A|^2+\varepsilon)^{\delta-1}(3-p)\frac{|A|^2}{r^2}-\delta (|A|^2+\varepsilon)^{\delta-1}|A|^4.
		\label{eq:pfSimons}
	\end{align}
	
	We then replace $\varphi$ by $(|A|^2+\varepsilon)^{\frac{\delta}{2}}\varphi$ in $\delta$-stability inequality to get
\begin{align*}
	&\delta\int_{ M} |A|^2(|A|^2+\varepsilon)^\delta \varphi^2\\
	\le{} & \int_{ M} \delta^2\varphi^2(|A|^2+\varepsilon)^{\delta-2}|A|^2|\nabla|A||^2+(|A|^2+\varepsilon)^\delta|\nabla \varphi|^2+\frac{1}{2}\left< \nabla\varphi^2,\nabla(|A|^2+\varepsilon)^\delta \right>\\
	={} & \int_{ M} \delta^2\varphi^2(|A|^2+\varepsilon)^{\delta-2}|A|^2|\nabla|A||^2+(|A|^2+\varepsilon)^\delta|\nabla \varphi|^2-\frac{1}{2}\varphi^2\Delta(|A|^2+\varepsilon)^\delta\\
	\le 
	{}& \int_{ M} (|A|^2+\varepsilon)^\delta|\nabla \varphi|^2+\delta(|A|^2+\varepsilon)^{\delta-2}((2-p-\delta)|A|^2-p\varepsilon)|\nabla|A||^2\varphi^2\\
	 {}& +\int_{ M} -\delta (3-p)(|A|^2+\varepsilon)^{\delta-1}\frac{|A|^2}{r^2}\varphi^2+\delta (|A|^2+\varepsilon)^{\delta-1}|A|^4\varphi^2.
\end{align*}
Here, we have used inequality \eqref{eq:pfSimons} for the second inequality.
To eliminate the term $(|A|^2+\varepsilon)^{\delta-2}|A|^2|\nabla|A||^2$, we set $p=2-\delta$.
It is straightforward to verify that $p < 1 + \frac{2}{n-1}$ provided $\delta > \frac{(n-2)^2}{4(n-1)}$ for any $n \ge 3$. This allows us to use \eqref{eq:SimonsIneq} and \eqref{eq:pfSimons}. Thus, we continue to estimate
\begin{align*}
	&\int_{ M} \delta \varepsilon|A|^2(|A|^2+\varepsilon)^{\delta-1}\varphi^2+p\delta \varepsilon (|A|^2+\varepsilon)^{\delta-2}|\nabla|A||^2\varphi^2\\
	&+\int_{ M} \delta (1+\delta)(|A|^2+\varepsilon)^{\delta-1}\frac{|A|^2}{r^2}\varphi^2\le \int_{ M} (|A|^2+\varepsilon)^\delta|\nabla \varphi|^2.
\end{align*}
Since the right-hand side $(|A|^2+\varepsilon)^\delta|\nabla\varphi|^2$ is uniformly bounded and the left-hand side is non-negative, we can let $\varepsilon \rightarrow 0$ to derive
\begin{equation}
	\delta(1+\delta)\int_{ M} r^{-2}|A|^{2\delta}\varphi^2\le \int_{ M} |A|^{2\delta}|\nabla\varphi|^2.
	\label{eq:pfDeltaStableCone}
\end{equation}

For some positive $\varepsilon$ to be chosen later, we set
\[
	\varphi=r^{\delta+\varepsilon}\max \{ 1,r \}^{1-\frac{n}{2}-2\varepsilon}.
\]
It is easy to verify that $\int_{M} r^{-2} |A|^{2\delta} \varphi^2 < \infty$, allowing us to plug it into the inequality \eqref{eq:pfDeltaStableCone} via an approximation argument (cf. \cite{Schoen1975curvature}). Consequently, we obtain,
\begin{align*}
	&\delta(1+\delta)\left[ \int_{ \left\{ r>1 \right\}} r^{2\delta-n-2\varepsilon}|A|^{2\delta}+\int_{ \left\{ r<1 \right\}} r^{2\delta-2+2\varepsilon}|A|^{2\delta} \right] \\
\le{}&
	 \left( 1+\delta-\frac{n}{2}-\varepsilon \right)^2 
	\int_{ \left\{ r>1 \right\}} r^{2\delta-n-2\varepsilon}|A|^{2\delta}+
	 \left(\delta+\varepsilon\right)^2\int_{ \left\{ r<1 \right\}} r^{2\delta-2+2\varepsilon}|A|^{2\delta}.
\end{align*}
Note that $\delta(1+\delta)>\delta^2$ is trivial, and we can check $\delta(1+\delta)>(1+\delta-\frac{n}{2})^2$ is equivalent to $\delta>\frac{(n-2)^2}{4(n-1)}$.
Hence, we can always choose $\varepsilon$ small enough such that the above inequality is impossible unless $|A|\equiv 0$.
\end{proof}

\subsection{Estimate of $L^{p}$ norm of second fundamental form}%
\label{sub:estimate_of_l_2alpha_norm_of_second_fundamental_form}

In this part, we demonstrate that if a $\delta$-stable minimal hypersurface is sufficiently close to a hyperplane or a flat cone, then the $L^{2\beta}$ norm of the second fundamental form is small for some $\beta \in (\frac{n-2}{n-1},1)$.

\begin{proposition}
	Let $n\ge 3$ and $\frac{n-2}{n}<\delta<1 $.
	Suppose $M$ is a immersed, two-sided, $\delta$-stable minimal hypersurface in $B^{n+1}_3(0)$ of $\mathbb{R}^{n+1}$ with $\mathcal{H}^{n-2}(\mathrm{sing}M\cap B^{n+1}_3(0))=0$.
	We also assume $\mathcal{H}^n(M\cap B^{n+1}_3(0))\le \Lambda$ for a positive constant $\Lambda.$
	Then there exists $\beta=\beta(n,\delta) \in (\frac{n-2}{n-1},1)$ such that
	\[
		\int_{ B^{n+1}_1(0)}|A|^{2\beta} \le C\left( \int_{ B^{n+1}_3(0)}|x_{n+1}|^2 \right)^\beta,
	\]
	holds for some $C=C(n,\delta,\Lambda)$.
	\label{prop:Alpha}
\end{proposition}

\begin{proof}

We define $g=\sqrt{1-(\nu\cdot e_{n+1})^2}$, where $\nu$ is the unit normal vector of $M$.
Then standard calculations show that (cf. \cite{SchoenSimon1981Regularity})
\[
	|\nabla g|^2 
	\le\frac{n-1}{n}|A|^2(1-g^2),
 \]
 and  
 \[g \Delta g =-\frac{|\nabla g|^2}{1-g^2}+|A|^2(1-g^2)\ge |A|^2\left( \frac{1}{n}-g^2 \right).
\]
We shall choose the test function $(g^\beta-k)^+\varphi$ in the $\delta$-stability inequality, where $\varphi$ is a locally Lipschitz function with compact support in $B^{n+1}_3(0)$ and vanishes in a neighborhood of $\mathrm{sing}M$.
Note that we can check $((g^\beta-k)^+)^2 \in C^1(M)\cap W^{2,\infty}_{\mathrm{loc}}(M)$. We denote $R_k=\{x\in M: g^\beta(x)> k\}.$
Then, direct computation yields
\begin{align*}
\int_M &|\nabla ((g^\beta-k)^+\varphi)|^2\\
	={} & \int_{R_k } (g^\beta-k)^2|\nabla \varphi|^2+\beta^2g^{2\beta-2}|\nabla g|^2\varphi^2-\frac{1}{2}\int_{R_k} \varphi^2\Delta (g^\beta-k)^2\\
	= 
	{}& \int_{R_k} (g^\beta-k)^2|\nabla \varphi|^2+\beta^2g^{2\beta-2}|\nabla g|^2\varphi^2\\
	  &-\int_{R_k} \varphi^2g^{2\beta-2}\left[ \left( \beta(2\beta-1)-\beta(\beta-1) \frac{k}{g^\beta} \right)|\nabla g|^2+\beta\left( 1-\frac{k}{g^\beta} \right) g \Delta g \right] \\
	= 
	{}& \int_{R_k} (g^\beta-k)^2|\nabla \varphi|^2+g^{2\beta-2}\varphi^2\left( 1-\frac{k}{g^\beta} \right)\left[ \beta(1-\beta)|\nabla g|^2-\beta g \Delta g \right] \\
	\le 
	{}& \int_{R_k} (g^\beta-k)^2|\nabla \varphi|^2\\
 +&\int_{R_k}g^{2\beta-2}\varphi^2\left( 1-\frac{k}{g^\beta} \right)|A|^2\left[ \frac{(n-1)\beta(1-\beta)}{n}(1-g^2)-\beta\left( \frac{1}{n}-g^2 \right) \right].
\end{align*}
For the left-hand side, we have
\begin{align*}
	{} & \int_{ R_k} \delta|A|^2(g^\beta-k)^2\varphi^2=\int_{ R_k} \delta g^{2\beta-2}|A|^2\varphi^2\left( 1-\frac{k}{g^\beta} \right)(g^2-kg^{2-\beta}).
\end{align*}
Thus, we obtain the following inequality
\begin{align*}
	&\int_{ R_k}g^{2\beta-2}|A|^2\varphi^2( 1-\frac{k}{g^\beta} )
	\left[ c_1g^2+c_2-\delta k g^{2-\beta} \right] 
	\le
	\int_{ R_k} (g^\beta-k)^2|\nabla \varphi|^2.
\end{align*}
where 
\[
	c_1=\delta-\beta + \frac{(n-1)\beta(1-\beta)}{n},\quad c_2=\frac{\beta}{n}\left( 1-(n-1)(1-\beta) \right).
\]

Since $\delta \in (\frac{n-2}{n},1)$, we can select $\beta \in (\frac{n-2}{n-1},1)$ such that
\[
	\beta - \frac{(n-1)\beta(1-\beta)}{n}<\delta,
\]
thus
\[c_1>0,\ \ \ \ \ \ \ c_2>0.\]
Since $g\leq 1$, we can find $k_0>0$ small enough such that for every $k \in (0,k_0)$, we have
\[
	\delta k g^{2-\beta}\le c_1 g^2+\frac{c_2}{2}.
\]
Therefore, we obtain
\[
	\frac{c_2}{2}\int_{ R_k} g^{2\beta-2}|A|^2\varphi^2(1-\frac{k}{g^\beta})\le \int_{ R_k} (g^\beta-k)^2|\nabla \varphi|^2.
\]
for any $k \in (0,k_0)$ and positive $c_2=c_2(\delta)$.
Now, we can take $k\rightarrow 0^+$ to obtain
\begin{equation}
	\int_{ M\cap \left\{ g>0 \right\}} g^{2\beta-2}|A|^2\varphi^2\le C_2\int_{M} g^{2\beta}|\nabla \varphi|^2,
	\label{eq:pfL2WithG}
\end{equation}
for some $C_2=C_2(\delta)$.
Now, by the approximate argument (cf. Lemma \ref{lem_boundedLpNorm}), together with the fact that $\mathcal{H}^{n-2}(\mathrm{sing}M\cap B^{n+1}_3(0))=0$, we know \eqref{eq:pfL2WithG} holds for any $\varphi$ which is locally Lipschitz and supported in $B^{n+1}_3(0)$.
Now, we use H\"older's inequality and \eqref{eq:pfL2WithG} to get
\begin{align}
	&\int_{ M\cap \left\{ g>0 \right\}\cap B^{n+1}_1(0)} |A|^{2\beta}\nonumber\\={}&
	\int_{ M\cap \left\{ g>0 \right\}\cap B^{n+1}_1(0)} |A|^{2\beta}g^{-2\beta(1-\beta)} g^{2\beta(1-\beta)}\nonumber \\
	\le 
	{}& \left( \int_{ M\cap \left\{ g>0 \right\}\cap B^{n+1}_1(0)}|A|^2g^{-2+2\beta} \right) ^{\beta}\left( \int_{ M\cap \left\{ g>0 \right\}\cap B^{n+1}_1(0)} g^{2\beta}  \right)^{1-\beta}\nonumber \\
	\le 
	{}& C \int_{M\cap B^{n+1}_2(0)} g^{2\beta}\le C \left( \int_{ B^{n+1}_2(0)} g^2 \right)^\beta (\mathcal{H}^n(B^{n+1}_2(0)\cap M))^{1-\beta}.
	\label{eq:pf2AlphaPositive}
\end{align}
We also need to estimate $\int_{ M\cap\left\{ g=0 \right\}\cap B^{n+1}_1(0)} |A|^{2\beta}$.
By replacing $\varphi$ by $g \varphi$ in the $\delta$-stability inequality \eqref{eq:deltaStability} and after a standard calculation involving integration by parts similar to the above calculations, we obtain
\[
	\int_{M} |A|^2\varphi^2\left( \frac{1}{n}-(1-\delta)g^2 \right)\le \int_{M} g^2|\nabla \varphi|^2.
\]
Note that $\varphi$ can be any locally Lipschitz function with compact support in $B^{n+1}_3(0)$ by the same approximate argument as above.
Using $\delta>\frac{n-2}{n}$, we derive
\begin{align}
	\frac{1}{2n}\int_{ M\cap\{ g<\frac{1}{2} \} } |A|^2\varphi^2\le{}& \int_{M} g^2|\nabla \varphi|^2+\int_{ M\cap\left\{ (1-\delta)g^2>\frac{1}{n} \right\}} |A|^2\left( (1-\delta)g^2-\frac{1}{n} \right)\varphi^2\nonumber\\
	 \le{}& \int_{M} g^2|\nabla \varphi|^2 + \int_{M\cap \left\{ g>\frac{1}{2} \right\}} |A|^2\varphi^2\nonumber\\ \le {}& C\int_{M} g^2|\nabla \varphi|^2 .\label{eq:pfg=0}
\end{align}
Here, we have used \eqref{eq:pfL2WithG} to estimate the term $\int_{ M\cap\{ g>\frac{1}{2} \}} |A|^2\varphi^2$.
Now, it follows from H\"older's inequality and \eqref{eq:pfg=0} that
\begin{equation}
	\int_{ M\cap\left\{ g=0 \right\}\cap B^{n+1}_1(0)} |A|^{2\beta}\le C \left( \int_{ M\cap B^{n+1}_2(0)} g^2 \right)^\beta (\mathcal{H}^n(B^{n+1}_2(0)\cap M))^{1-\beta},
	\label{eq:pf2Alpha0}
\end{equation}
for some $C=C(n,\delta)$.
Note that by minimality, we can obtain
\[
	\int_{ B^{n+1}_2(0) \cap M} g^2=\int_{ B^{n+1}_2(0) \cap M} |e_{n+1}^\top |^2\le C \int_{ B^{n+1}_3(0) \cap M}|x_{n+1}|^2
\]
for some $C=C(n)$, where $e_{n+1}^\top $ means the projection of $e_{n+1}$ to the tangent space of $M$.
To see this, we insert $\varphi^2x_{n+1}e_{n+1}$ in to the first variation formula to get
\begin{align*}
	\int_{ M} \varphi^2|e_{n+1}^\top |^2={}& -2\int_{ M} \varphi \left< \nabla\varphi,  e_{n+1}^\top\right> x_{n+1}\\
	\le{}& \frac{1}{2}\int_{ M} \varphi^2|e_{n+1}^\top|^2+2\int_{ M} |\nabla \varphi|^2x_{n+1}^2,
\end{align*}
by the Cauchy-Schwarz inequality.

Together with \eqref{eq:pf2AlphaPositive} and \eqref{eq:pf2Alpha0}, we can obtain the desired estimate.
\end{proof}

With the help of $\varepsilon$-regularity theorem, we can prove the following results.
\begin{proposition}\label{prop:closePlane}
	
	Let $n\ge 3$, $\delta>\frac{n-2}{n}$ and $\bar{n}<n-4+\frac{4}{n}$.
	Suppose $M_j$ is a sequence of immersed, two-sided, $\delta$-stable minimal hypersurfaces in $B^{n+1}_4(0)$ with $\mathrm{dim}(\mathrm{sing}(M_j)\cap B^{n+1}_4(0))\le \bar{n}$, such that $M_j$ converge (as varifolds) to $q|P\cap B^{n+1}_4(0)|$ as $j\rightarrow \infty$ where $P$ is a hyperplane and $q$ is a positive integer.
	Then,
	\[
		\lim_{j\rightarrow \infty} \sup_{B^{n+1}_{\frac{1}{2}}(0)\cap M_j}|A_{M_j}|=0.
	\]
	Moreover, $\mathrm{sing} M_j\cap B^{n+1}_{\frac{1}{4}}(0)=\emptyset$ and $M_j\cap B^{n+1}_{\frac{1}{4}}(0)$ has exact $q$ connected components for $j$ large enough, and each component of $M_j\cap B^{n+1}_{\frac{1}{4}}(0)$ converge smoothly to $P$ in $B^{n+1}_{\frac{1}{4}}(0)$ as smooth immersions.
\end{proposition}
\begin{proof}
	We suppose $P=\left\{ x_{n+1}=0 \right\}$.
	By the monotonicity formula and the lower semicontinuity of mass, we have
	\[
		\lim_{j\rightarrow \infty} \sup_{B^{n+1}_3(0)\cap M_j}|x_{n+1}|=0.
	\]
	Now, we can apply Proposition \ref{prop:Alpha}, we know
	\[
		\lim_{j\rightarrow \infty} \int_{ B^{n+1}_1(0)\cap M_j}|A_{M_j}|^{2\beta}=0
	\]
	for some $\beta \in (\frac{n-2}{n-1},1)$.

	Noting that $\mathcal{H}^n(M_j\cap B^{n+1}_3(0))\rightarrow 3^nq\omega_n$ as $j\rightarrow \infty$, using H\"older's inequality, we have
	\[
		\lim_{j\rightarrow \infty} \int_{ B^{n+1}_1(0)\cap M_j} |A_{M_j}|^{2\alpha}=0
	\]
	for any $\alpha \in (\frac{n-2}{n},\min \{ \delta,\bar{n},\frac{n-2}{n-1},1 \})$.
	Now, we can apply $\varepsilon$-regularity theorem (Theorem \ref{thm:epsilonRegularity}) to get
	\begin{equation}
		\lim_{j\rightarrow \infty} \sup_{B^{n+1}_{\frac{1}{2}}(0)\cap M_j}|A_{M_j}|=0.
		\label{eq:pfLimitAconverges}
	\end{equation}
	Next, let us denote $S_j=P(\mathrm{sing}(M_j))$ where $P$ is the projection to $\left\{ x_{n+1}=0 \right\}$.
	Then, we know the projection $P$ gives us a covering map from $M_j\cap (B_{\frac{1}{4}}^n(0)\setminus S_j) \times \mathbb{R}$ to $(B_{\frac{1}{4}}^n(0)\backslash S_j) \times \left\{ 0 \right\}$, and the covering degree is $q$ by \eqref{eq:pfLimitAconverges} for $j$ large enough.
	Since $B_{\frac{1}{4}}^n(0)\backslash S_j$ is simply connected, as $\mathrm{dim}(S_j)\le \bar{n}$, we know $M_j\cap (B_{\frac{1}{4}}^n(0)\setminus S_j) \times \mathbb{R}$ has exact $q$ connected components and each component of $M_j\cap (B_{\frac{1}{4}}^n(0)\setminus S_j) \times \mathbb{R}$ can be written as a graph over $B_{\frac{1}{4}}^n(0)\backslash S_j$ of a smooth function.
	By the removable singularity theorem (cf. \cite{DeGigoriSingolarita1965removableSing, Simon1977removableSing}), we know such function can be extended to a smooth function on $B_{\frac{1}{4}}^n$ which solves the minimal surface equation and hence, $\mathrm{sing}M_j \cap B_{\frac{1}{4}}^n(0)\times \mathbb{R}=\emptyset$ and it can be decomposed into $q$ connected components, each of which converges smoothly to $B_{\frac{1}{4}}^n(0)\times \left\{ 0 \right\}$ as smooth immersions by the standard PDE theory.
\end{proof}

We consider $\boldsymbol{C}$ to be the flat cone in $\mathbb{R}^{n+1}$ defined by the union of hyperplanes and half-hyperplanes.
Explicitly, we can represent $\boldsymbol{C}$ as
\[
	\boldsymbol{C}:=\sum_{i =1}^{N_1}p_i|P_i|+\sum_{i=1}^{N_2}q_i|H_i|,
\]
for $\{ p_i \}_{i=1}^{N_1}$, $\{ q_i \}_{i=1}^{N_2}\subset \mathbb{N}$.
Here $\{ P_i \}_{i=1}^{N_1}$ are distinct hyperplanes and $\{ H_i \}_{i=1}^{N_2}$ are distinct half-hyperplanes such that $0 \in P_i$ for each $1\le i\le N_1$, $0 \in \bar{H}_i$ for each $1\le i\le N_2$, and $H_j\nsubseteq P_i$ for each $1\le i\le N_1$ and $1\le j\le N_2$.

We denote the singular set of $\boldsymbol{C}$ in the embedded sense as $T(\boldsymbol{C})$, which is precisely defined as:
\[
	T(\boldsymbol{C}):=\left\{ x \in \mathrm{spt}\|\boldsymbol{C}\|: \mathrm{spt}\|\boldsymbol{C}\| \text{ is not a part of hyperplane near }x\right\}.
\]

We denote $T_\tau(\boldsymbol{C})$ as the $\tau$-neighborhood of $T(\boldsymbol{C})$ for $\tau>0$.

\begin{proposition}
	Let $n\ge 3$, $\delta>\frac{n-2}{n}$, and $\bar{n}<n-4+\frac{4}{n}$.
	Suppose $M_j$ is a sequence of smooth, immersed, two-sided $\delta$-stable minimal hypersurfaces in $B^{n+1}_4(0)$ with $\mathrm{dim}(\mathrm{sing}(M_j)\cap B^{n+1}_4(0))\le \bar{n}$, such that $M_j$ converge (as varifolds) to $\boldsymbol{C}\lfloor(B^{n+1}_4(0))$ as $j\rightarrow \infty$.
	Then,
	\[
		\lim_{j\rightarrow \infty} \sup_{B^{n+1}_{\frac{1}{2}}(0)\cap M_j}|A_{M_j}|=0.
	\]

	In particular, $\boldsymbol{C}$ is a sum of hyperplanes with multiplicity, $\mathrm{sing} M_j\cap B^{n+1}_{\frac{1}{4}}(0)=\emptyset$, and $M_j\cap B^{n+1}_{\frac{1}{4}}(0)$ converge smoothly to $\boldsymbol{C}$ in $B^{n+1}_{\frac{1}{4}}(0)$ as immersions with $q$ connected components for $q=\Theta(\|\boldsymbol{C}\|,0)$.
	\label{prop:closeClassicalCone}
\end{proposition}
\begin{proof}
	For each fixed $\tau>0$, we know $M_j\cap B^{n+1}_3(0)$ converge to $\boldsymbol{C}\lfloor (B^{n+1}_3(0)\backslash T_\tau(\boldsymbol{C}))$ smoothly as $j\rightarrow \infty$ by Proposition \ref{prop:closePlane}.
	For each $\tau>0$, by the proof of Proposition \ref{prop:closePlane}, we know
	\[
		\lim_{j\rightarrow \infty} \int_{ B^{n+1}_2(0)\cap M_j\backslash T_\tau(\boldsymbol{C})} |A_{M_j}|^{2\alpha}=0.
	\]
	Now, we take $k=0$, $\phi$ to be a non-negative cut-off function supported in $B^{n+1}_{\frac{3}{2}}(0)$, equal to $1$ in $B^{n+1}_1(0)$, and $|\nabla \phi|\le 4$ in Lemma \ref{lem:weakPoincare}, together with the Michael-Simon's inequality, we have
	\begin{align}
		&\int_{ M_j\cap B^{n+1}_1(0)}|A_{M_j}|^{\frac{2\alpha n}{n-2}}\le 
		C \left( \int_{ M_j\cap B^{n+1}_{\frac{3}{2}}(0)} |A_{M_j}|^{2\alpha} \right)^{\frac{n}{n-2}}\nonumber \\
		\le{}& C \left( \int_{ M_j\cap B^{n+1}_{\frac{3}{2}}(0)} |A_{M_j}|^{2} \right)^{\frac{n\alpha}{n-2}}(\mathcal{H}^n(M_j\cap B^{n+1}_{\frac{3}{2}}(0)))^{\frac{n(1-\alpha)}{n}},
	\label{eq:pfSecondUniBound}
	\end{align}
	for some $C=C(n)$.
	Note that the $\delta$-stability condition implies the right-hand side of \eqref{eq:pfSecondUniBound} is uniformly bounded by Lemma \ref{lem_boundedLpNorm}.
	Hence,
	\[
		\sup_{j>0}\int_{ M_j\cap B^{n+1}_1(0)} |A_{M_j}|^{\frac{2\alpha n}{n-2}}<\infty.
	\]
	
	By H\"older's inequality, we have
	\begin{align*}
		&\int_{ M_j\cap T_\tau(\boldsymbol{C}) \cap B^{n+1}_1(0)}|A_{M_j}|^{2\alpha}\\
		\le{} & \left( \int_{ M_j\cap T_\tau(\boldsymbol{C})\cap B^{n+1}_1(0)}|A_{M_j}|^{\frac{2\alpha n}{n-2}}  \right)^{\frac{n-2}{n}}\mathcal{H}^n(M_j\cap T_\tau(\boldsymbol{C})\cap B^{n+1}_1(0))^{\frac{2}{n}}.\\
	\end{align*}

	By a standard covering argument with monotonicity formula, we know that
	\[
		\mathcal{H}^n(M_j\cap T_\tau(\boldsymbol{C}) \cap B^{n+1}_2(0))\le C\tau,
	\]
	for some $C=C(n,\Theta(\|\boldsymbol{C}\|,0))$ for $j$ large enough.
	Hence, we have
	\begin{equation}
		\lim_{j\rightarrow \infty} \int_{ B^{n+1}_1(0)\cap M_j} |A_{M_j}|^{2\alpha}=\lim_{j\rightarrow \infty} \int_{ B^{n+1}_1(0)\cap M_j\cap T_\tau(\boldsymbol{C})} |A_{M_j}|^{2\alpha}\le C\tau^{\frac{2}{n}},
		\label{eq:pfAlphaTau}
	\end{equation}
	for some $C<\infty$ which is independent of $\tau$.
	Since the left-hand side of \eqref{eq:pfAlphaTau} is independent of $\tau$, by the arbitrariness of $\tau$, we obtain
	\[
		\lim_{j\rightarrow \infty} \int_{ B^{n+1}_1(0)\cap M_j} |A_{M_j}|^{2\alpha}=0.
	\]
	Thus, we apply $\varepsilon$-regularity theorem (Theorem \ref{thm:epsilonRegularity}) to conclude that $$\lim_{j\rightarrow \infty} \sup_{B^{n+1}_{\frac{1}{2}}(0)\cap M_j}|A_{M_j}|=0,$$
 which implies that each connected component of $M_j\cap B^{n+1}_{\frac{1}{2}}(0)$ converge to a hyperplane in the varifold sense. Furthermore, by Proposition \ref{prop:closePlane}, for $j$ large enough, $\mathrm{sing}M_j\cap B^{n+1}_{\frac{1}{4}}(0)$ is empty, and $M_j\cap B^{n+1}_{\frac{1}{4}}(0)$ has exact $q$ connected components, each of which converges smoothly to a hyperplane in $B^{n+1}_{\frac{1}{4}}(0)$.
\end{proof}

\subsection{Finish the proof of Theorem \ref{thm_regularity_and_compactness_theorem}}%
\label{sub:finish_the_proof_of_theorem_thm}

\begin{proof}
	[Proof of Theorem \ref{thm_regularity_and_compactness_theorem}]
	By Allard's compactness theorem (cf. \cite{Allard1972, simon1983lectures}), we obtain a stationary integral varifold $V$ of $B^{n+1}_4(0)$ such that up to a subsequence, $|M_j|$ converges to $V$ in the sense of varifolds.
	Let $S=\mathrm{sing}\|V\|$ be the singular point set of $V$.
	We need to analyze the tangent cone $\boldsymbol{C}$ of $V$ at $x_0\in S\cap B^{n+1}_{\frac{1}{2}}(0)$.
	Indeed, we have the following lemma.
	\begin{lemma}
		For any $\boldsymbol{C} \in \mathrm{VarTan}(V,x_0)$ for $x_0 \in S \cap B^{n+1}_{\frac{1}{2}}(0)$, we can write $\boldsymbol{C}=\boldsymbol{C}'\times \mathbb{R}^{n-p}$ for some $p\ge n_\delta$.
		\label{lem:coneDimReduction}
	\end{lemma}
	\begin{proof}
		[Proof of Lemma \ref{lem:coneDimReduction}]
		For any cone $\boldsymbol{C}$, we write $\mathcal{S}(\boldsymbol{C})$ (the spine of $\boldsymbol{C}$) to be the linear subspace containing all $x\in \mathbb{R}^{n+1}$ such that $\boldsymbol{C}$ is invariant under the translation along the line spanned by $x$.
		For any $x_0 \in S$, we introduce the notion of iterated tangents of $V$ at $x_0$ as follows.
		We say a collection of cones $\left\{ \boldsymbol{C}_1,\boldsymbol{C}_2,\cdots ,\boldsymbol{C}_N \right\}$ is iterated tangents of $V$ at $x_0$ if $\boldsymbol{C}_1$ is the tangent cone of $V$ at $x_0$, and $\boldsymbol{C}_{j+1}$ is the tangent cone of of $\boldsymbol{C}_j$ at $x_j \in \mathrm{sing}\|\boldsymbol{C}_j\|\backslash \mathcal{S}(\boldsymbol{C}_j)$ for $1\le j\le N-1$.
		Moreover, by the standard argument (cf. \cite{SchoenSimon1981Regularity,Wickramasekera2008regularityCompactMult2,bellettini}), we can make sure the iterated tangents satisfy the following properties:
		\begin{enumerate}
			\item Each $\boldsymbol{C}_j$ is not smoothly immersed (i.e., $\mathrm{sing}\|\boldsymbol{C}_j\|\neq \emptyset$).
			\item $\mathrm{dim}(\mathcal{S}(\boldsymbol{C}_{j+1}))>\mathrm{dim}(\mathcal{S}(\boldsymbol{C}_{j}))$ for each $j=1,2,\cdots ,N-1$.
			\item $\boldsymbol{C}_N=\boldsymbol{C}'\times \mathbb{R}^{\mathrm{dim}(\mathcal{S}(\boldsymbol{C}_N))}$ where $\boldsymbol{C}'\backslash \left\{ 0 \right\}$ is a smooth immersed cone after a suitable rotation in $\mathbb{R}^{n+1}$.
			\item For each $1\le j\le N$, we can find a sequence of points $\left\{ y_k \right\}$ with $y_k\rightarrow x_0$, a sequence of positive real numbers $\left\{ r_k \right\}$ with $r_k\rightarrow 0^+$ as $k\rightarrow \infty$, such that $\eta_{y_k,r_k}(M_k)$ converges to $\boldsymbol{C}_j$ in the sense of varifolds and the convergence is smooth away from the singular set of $\boldsymbol{C}_j$ by Proposition \ref{prop:closePlane} and Proposition \ref{prop:closeClassicalCone}.
		\end{enumerate}
		In particular, the fourth condition implies that the smooth immersed part of $\boldsymbol{C}_j$ is stable, and the second condition implies $N$ is a finite number.

		Now, let us determine the dimension of $\boldsymbol{C}'$.
		Note that $\boldsymbol{C}_N$ cannot be a hyperplane by the first condition.

		If the dimension of $\boldsymbol{C}'$ is one, then $\boldsymbol{C}_N$ is the sum of distinct half-hyperplanes with multiplicity.
		But by Proposition \ref{prop:closeClassicalCone}, we know $\boldsymbol{C}_N$ is a sum of hyperplanes with multiplicity, which contradicts the first condition.

		Therefore, we know $\boldsymbol{C}'$ has dimension at least two.
		But the fourth condition implies that $\boldsymbol{C}'$ is a smooth immersed $\delta$-stable cone away from $\left\{ 0 \right\}$, and hence, $\boldsymbol{C}'$ has dimension at least $n_\delta$ by Proposition \ref{thm:deltaCone}.

		Hence, by the second condition, we know $\mathrm{dim}(\boldsymbol{C})\ge n-n_\delta$ for any $\boldsymbol{C}\in \mathrm{VarTan}(V,x_0)$, and the lemma follows.
	\end{proof}
	Now, let us finish the proof of Theorem \ref{thm_regularity_and_compactness_theorem}.
	If $n\ge n_\delta$, then using Lemma \ref{lem:coneDimReduction} we can apply Federer's dimension reducing principle (see e.g. \cite[Appendix A]{simon1983lectures} to get $\mathrm{dim} (\mathrm{sing}(\|V\|)\cap B^{n+1}_{\frac{1}{2}}(0))\le n-n_\delta$, and when $n<n_\delta$, we can directly apply Lemma \ref{lem:coneDimReduction} to get $\mathrm{sing}(\|V\|)\cap B^{n+1}_{\frac{1}{2}}(0)=\emptyset$.
	At last, we need to show when $n=n_\delta$, $\mathrm{sing}(\|V\|)\cap B^{n+1}_{\frac{1}{2}}(0)$ is discrete.

	We argue by contradiction.
	Suppose we can find a sequence of points $\left\{ x_j \right\}\subset \mathrm{sing}(\|V\|)\cap B^{n+1}_{\frac{1}{2}}(0)$ such that $x_j\rightarrow x_0\in \mathrm{sing}(\|V\|)\cap B^{n+1}_{\frac{1}{2}}(0)$ for some $x_0 \in \mathrm{sing}(\|V\|)\cap B^{n+1}_{\frac{1}{2}}(0)$.
	Up to a subsequence, we can assume $(\eta_{x_0,\rho_i})_{\#}V$ converges to $\boldsymbol{C} \in \mathrm{VarTan}(V,x_0)$ where $\rho_i=|x_i-x_0|$, and assume $y=\lim_{i\rightarrow \infty} \frac{x_i-x_0}{\rho_i}\neq 0$.
	Note that since $y$ is a regular point of $\boldsymbol{C}$ by Lemma \ref{lem:coneDimReduction}, we know $(\eta_{x,\rho_i})_{\#}V$ is a smooth immersion in a neighborhood of $y$ for $i$ large enough by Proposition \ref{prop:closeClassicalCone}, which contradicts the fact that $x_i$ is a singular point of $V$.
	This finishes the proof of Theorem \ref{thm_regularity_and_compactness_theorem}.
\end{proof}
\begin{proof}
    [Proof of Corollary \ref{thm:mainAreaGrowth} and Corollary \ref{thm:mainCurvatureEst}]

    Suppose the Generalized Bernstein Theorem fails.
	Then, for $n\ge 3$, there exists complete minimal, two-sided, $\delta$-stable immersed hypersurface in $\mathbb{R}^{n+1}$ with $\delta>\max \left\{ \frac{n-2}{n},\frac{(n-2)^2}{4(n-1)} \right\}$ with the volume growth estimate
	\[
		\mathcal{H}^{n}(M\cap B^{n+1}_R(0))\le \Lambda R^n,
	\]
	for some $\Lambda \in (0,+\infty)$. 
	Besides, there exists a point, says $0$, such that $|A|(0)>0$.

	Consider a scaling sequence $M_j:=\frac{1}{j}M$ for $j>0$ and let $A_j$ denote the second fundamental form of $M_j$.
	By taking a subsequence (and an abuse of notation), we assume that $M_j$ converges to an integral stationary $n$-varifold $V$ in the sense of varifolds with
	\begin{equation}
		\lim_{j\rightarrow \infty} |A_j|(0)=+\infty.
		\label{eq:pfUnboundedII}
	\end{equation}
	Moreover, since $V$ is a blow-down limit of $M_j$, we know $V$ is a cone with the vertex at $0$ by the monotonicity formula and a standard argument in geometric measure theory.

	By Theorem \ref{thm_regularity_and_compactness_theorem}, $V$ is a sum of hyperplanes with multiplicity since $V$ is a cone.
	By applying Proposition \ref{prop:closeClassicalCone}, we obtain
	\begin{align*}
		\lim_{j\rightarrow \infty} \sup_{B^{n+1}_{\frac{1}{4}}(0)\cap M_j}|A_j|={} & 0
	\end{align*}
	which contradicts \eqref{eq:pfUnboundedII}.
	
	Hence, Corollary \ref{thm:mainAreaGrowth} follows.
	The proof of Corollary \ref{thm:mainCurvatureEst} is a direct consequence of Corollary \ref{thm:mainAreaGrowth} by a standard blow-up argument.
\end{proof}

\section{Conformal metric and $\mu$-bubble}
\label{sec:area_growth_estimate}

\subsection{Conformal change of the metric}\label{conformalchagneofthemetric}
Let $(N^n,g)$ be a compact manifold with smooth boundary. Consider the manifold $(N^n,\bar{g})$ where $\bar{g}=u^{\frac{2\alpha}{n-1}} g$ and $\alpha>0$, it is well-defined since $u$ is positive. Let $h$ be a smooth function on $N$ (to be determined). Choose a Caccioppoli set $\Omega_0$ with smooth interior boundary $\partial \Omega_0\subset \mathring{N}$ and $\partial N\subset \Omega_0$. We consider a prescribed mean curvature problem, that is, to find a Caccioppoli set $\Omega$ with smooth boundary such that $\Omega_0\subset \Omega$ and the mean curvature $\bar{H}$ of the interior boundary $\partial \Omega$ satisfies $\bar{H}=hu^{-\frac{1}{n-1}\alpha}.$ 

Consider the following functional 
\[\mathcal{A}_k(\Omega)=\int_{\partial\Omega} d\bar{\mathcal{H}}^{n-1}-\int_N(\chi_\Omega-\chi_{\Omega_0})hu^{-\frac{1}{n-1}\alpha} d\bar{\mathcal{H}}^n\]
for all Caccioppoli sets $\Omega$ in $N$ with $\Omega\Delta\Omega_0$ compactly contained in $N$. Here the Hausdoff measure is with respect to the conformal metric $\bar{g}$. Suppose that $h$ is chosen such that the minimizer of this functional exists and regularity is ensured. 

The first variation is
\[\mathcal{A}_k'=\int_{\partial\Omega}(\bar{H}-hu^{-\frac{1}{n-1}\alpha})\bar{\psi}.\]
The critical point of $\mathcal{A}_k$ satisfies
\[\bar{H}_{\partial \Omega}=hu^{-\frac{1}{n-1}\alpha}.\]
The mean curvature is defined as $\sum \langle\bar{D}_{\bar{e_i}} \bar{e_i},-\bar{\nu}\rangle$.
And the second derivative at the critical point is
\begin{align*}
    \mathcal{A}_k''=\int_{\partial \Omega} (-\bar{\Delta}\bar{\psi}-|\bar{A}|^2\bar{\psi}-\bar{Ric}(\bar{\nu},\bar{\nu})\bar{\psi})\bar{\psi}-\bar{\psi}^2\bar{D}_{\bar{\nu}}(hu^{-\frac{1}{n-1}\alpha})\geq 0
\end{align*}
where $\bar{\psi}$ is the speed of the variation. The minimizer of $\mathcal{A}_k$  is called \textit{warped $\mu$-bubble}.

\subsection{Variations of functional $\mathcal{A}_k$}\label{variationsofA}
The functional $\mathcal{A}$ can also be written (by letting $k=\alpha$) as 
\[\mathcal{A}_k(\Omega)=\int_{\partial\Omega} u^{k}d\mathcal{H}^{n-1}-\int_N(\chi_\Omega-\chi_{\Omega_0})hu^{k} d\mathcal{H}^n.\]

The first and second variations for $\mathcal{A}_k$ can be calculated directly; see \cite{chodoshlisoapbubble} for detailed calculation. We list them as follows (see \cite{hong24,hongyan24} for their applications)

\begin{lemma}[The first variation]\label{first}
    If $\Omega_t$ is a smooth $1$-parameter family of regions with $\Omega_0=\Omega$ and the normal speed at $t=0$ is $\psi$, then
    \begin{equation*}
        \frac{d}{dt}\mA_{k}(\Omega)=\int_{\Sigma_t}\left(Hu^k+\lp\nabla^{N}u^k,\nu\rp-hu^k\right)\psi d\mH^{n-1}
    \end{equation*}
    where $\nu$ is the outwards pointing unit normal and $H$ is the scalar mean curvature of $\partial \Omega_t$. In particular, a $\mu$-bubble $\Omega$ satisfies
    \begin{equation*}
        H=-ku^{-1}\lp\nabla^{N}u,\nu\rp+h,
    \end{equation*}
    along $\partial\Omega$.
\end{lemma}
Denote $\Sigma=\partial\Omega.$
\begin{lemma}[The second variation]\label{second}
 If $\Omega_t$ is a smooth $1$-parameter family of regions with $\Omega_0=\Omega$ and the normal speed at $t=0$ is $\psi$, then
    \begin{equation*}
        \begin{split}
            \frac{d^2}{d t^2}\bigg|_{t=0}\mA_{k}(\Omega)={}&\int_{\Sigma}|\nabla^{\Sigma}\psi|^2 u^k-\left(|A_{\Sigma}|^2+\Ric^N(\nu,\nu)\right)\psi^2 u^k\\
            &+\int_M k(k-1)u^{k-2}\left(\nabla^N_{\nu}u\right)^2\psi^2 \\
            &+\int_{\Sigma}ku^{k-1}\left(\Delta_N u-\Delta_{\Sigma} u\right)\psi^2-\left(\nabla_{\nu}(u^kh)\right)\psi^2.
        \end{split}
    \end{equation*}
\end{lemma}

Let $M$ be an immersion in Proposition \ref{thm:mainAreaGrowth} which is $\delta$-stable. By translation, we assume the image passes the origin. Let $g$ be the pull-back metric on $M$. Denote $\tilde{g}=r^{-2}g$ be the new metric on $N=M\setminus \{0\}$ where $r(x)$ is the distance function to the origin. Notice that the conformal change by $r^{-2}$ turns the ambient space $  \mathbb{R}^{n+1}$ to a manifold $(\mathbb{R}\times\mathbb{S}^n, dt^2+g_{\mathbb{S}^n})$ which has positive scalar curvature. Denote $d\mu, d\tilde{\mu}$ the volume elements of $M$ with respect to $g$ and $\tilde{g}$, respectively. 

We first show that the $\delta$-stability inequality in the metric $\tilde{g}$ in the following. The proof follows in the same lines as in \cite[Proposition 3.10]{chodoshliR5}.
\begin{proposition}
    \[
	\int_{ N} |\tilde{\nabla} \psi|^2_{\tilde{g}}d\tilde{\mu}\ge \int_{ N} \left( \delta r^2|A_N|^2-\frac{n(n-2)}{2}+\frac{n^2-4}{4}|d r|^2 \right)\psi^2 d\tilde{\mu}
\]
for any $\psi\in C^\infty_c(N)$, where $A_N$ is the second fundamental form of $(N,g)$ in $\mathbb{R}^{n+1}$ and $|\cdot|$ is with respect to the induced metric $g$.
\end{proposition}

Thus there exists a positive function $u$ on $N$ such that 
\begin{equation}\label{conformaljacobi}
    \tilde{\Delta} u+(\delta r^2|A_N|^2-\frac{n(n-2)}{2}+\frac{n^2-4}{4}|d r|^2 )u=0.
\end{equation}
We use this function $u$ in the definition of $\mathcal{A}$ in section \ref{conformalchagneofthemetric} or section \ref{variationsofA}. 

\subsection{$\mu$-bubble in $M^3\subset \mathbb{R}^4$} Let us assume $n=3$ in this subsection.

\begin{theorem}\label{diameterandvolumeboundn=3}
    Let $(N_0,\tilde{g})$ be a compact subset of $(N,\tilde{g})$ with connected boundary.  Suppose there exists $p\in N_0$ such that $d_{\tilde{g}}(p,\partial N_0)>4\sqrt{2}\pi. $ Then there exists a connected relative open set $\Omega$ containing $\partial N_0$, $\Omega\subset B_{4\sqrt{2}\pi}^{\tilde{g}}(\partial N_0)$ (Tubular neighborhood under metric $\tilde{g}$), such that each connected component of $\partial \Omega\setminus \partial N_0$ is a 2-sphere with area bounded by $16\pi$ and diameter bounded by $4\pi$ under the metric $\tilde{g}$.
\end{theorem}
\begin{proof}We use $\tilde{}$ to denote terms defined or calculated with respect to the metric $\tilde{g}$.
First we claim that there exists a function $h$ such that (1) the minimizer of $\mathcal{A}_k$ exists and contains $\partial N_0$, denote it by $\Omega$; (2) $h$ satisfies 
\[\frac{1}{2}+h^2-2|\tilde{\nabla} h|\geq 0.\]
This can be achieved by standard arguments (see \cite{chodosh-li-stryker,chodoshliR4anisotropic,chodoshliR5,hong24}). The distance assumption $d_{\tilde{g}}(p,\partial N_0)>4\sqrt{2}\pi$ is used here.

    Since $\Omega$ is the minimizer of $\mathcal{A}_k$ and let $\Sigma$ be one of components in $\partial\Omega$, then by Lemma \ref{second},
    \begin{align*}
        &\int_{\Sigma}|\tilde{\nabla}^{\Sigma}\psi|^2 u^k-ku^{k-1}\psi^2\tilde{\Delta}_\Sigma u\\\geq{}& \int_\Sigma -ku^{k-1} \psi^2\tilde{\Delta}_N u+
\left(|\tilde{A}_{\Sigma}|^2+\tilde{\Ric}^N(\tilde{\nu},\tilde{\nu})\right)\psi^2 u^k\\
            &+\int_\Sigma -k(k-1)u^{k-2}\left(\tilde{\nabla}^N_{\nu}u\right)^2\psi^2 +\left(\tilde{\nabla}^N_{\nu}(u^kh)\right)\psi^2.
    \end{align*}
    By the critical equation we have
    \[\tilde{H}_\Sigma^2=k^2u^{-2}\langle \tilde{\nabla}^N u,\nu\rangle_{\tilde{g}}^2+h^2-2khu^{-1}\langle \tilde{\nabla}^N u,\nu\rangle_{\tilde{g}}.\]
    Then
    \begin{align*}
       &\int_{\Sigma}|\tilde{\nabla}^{\Sigma}\psi|^2 u^k-ku^{k-1}\psi^2\tilde{\Delta}_\Sigma u\\\geq{}& \int_\Sigma \left(-k \frac{\tilde{\Delta}_N u}{u}+
|\tilde{A}_{\Sigma}|^2+\tilde{\Ric}^N(\tilde{\nu},\tilde{\nu})-\frac{1}{2}\tilde{H}_\Sigma^2\right)\psi^2 u^k\\
            &+\int_\Sigma \left(\frac{k^2}{2}-k(k-1)\right)u^{k-2}\left(\tilde{\nabla}^N_{\nu}u\right)^2\psi^2 +\frac{1}{2}h^2\psi^2u^k+(\tilde{\nabla}^N_{\nu}h)u^k\psi^2. 
    \end{align*}
    Replace $\psi$ by $u^{-k/2}\psi$ and let $k\leq 2$, we obtain
    \begin{align*}
       \int_{\Sigma}\frac{4}{4-k}|\tilde{\nabla}^\Sigma\psi|^2&\geq \int_\Sigma \left(-k \frac{\tilde{\Delta}_N u}{u}+
|\tilde{A}_{\Sigma}|^2+\tilde{\Ric}^N(\tilde{\nu},\tilde{\nu})-\frac{1}{2}\tilde{H}_\Sigma^2\right)\psi^2 \\
            &+\int_\Sigma \left(\frac{1}{2}h^2+(\tilde{\nabla}^N_{\nu}h)\right)\psi^2. 
    \end{align*}
    The remaining part is to show the following
    \begin{equation}\label{maininequality}
        -k \frac{\tilde{\Delta}_N u}{u}+
|\tilde{A}_{\Sigma}|^2+\tilde{\Ric}^N(\tilde{\nu},\tilde{\nu})-\frac{1}{2}\tilde{H}_\Sigma^2+\frac{1}{2}\tilde{R}_\Sigma\geq \epsilon_0
    \end{equation}
holds for a positive constant $\epsilon_0.$
By Gauss equation of $\Sigma\subset (N,\tilde{g})$, we have the following well-known equation
\[\tilde{R}_N-2\tilde{\Ric}^N(\tilde{\nu},\tilde{\nu})=\tilde{R}_\Sigma-\tilde{H}_\Sigma^2+|\tilde{A}_\Sigma|^2.\]
Thus \eqref{maininequality} turns to
\[-k \frac{\tilde{\Delta}_N u}{u}+\frac{1}{2}\tilde{R}_N+\frac{1}{2}|\tilde{A}_\Sigma|^2\geq \epsilon_0.\]
Note that all the calculations till now are with respect to the metric $\tilde{g}$, so is every term in the above inequality. We now transform them back to terms with respect to $g.$ By conformal-changing formula,
\begin{align*}
    \tilde{R}_N&=r^2(R^g_N-4\Delta_g\ln \frac{1}{r}-2|\nabla^g\ln (\frac{1}{r})|^2)\\
    &=r^2 R^g_N-4r^2\left(-\frac{\Delta r}{r}+\frac{5}{2}\frac{|d r|^2}{r^2}\right)\\
    &=r^2 R^g_N-4r^2\left(-\frac{3}{r^2}-\frac{H_N(x\cdot\nu)}{r}+\frac{5}{2}\frac{|d r|^2}{r^2}\right)\\
    &=r^2 R^g_N+12-10|d r|^2.
\end{align*}
By the Gauss equation of $N\subset \mathbb{R}^4$, we have
\[R^g_N=-|A_N|^2.\]
Hence, due to \eqref{conformaljacobi},
\begin{align*}
   -k \frac{\tilde{\Delta}_N u}{u}+\frac{1}{2}\tilde{R}_N&\geq k(\delta r^2|A_N|^2-\frac{3}{2}+\frac{5}{4}|d r|^2)+\frac{1}{2}(-r^2|A_N|^2+12-10|d r|^2)\\
   &= (k\delta-\frac{1}{2})r^2|A_N|^2+6-\frac{3}{2}k+(\frac{5}{4}k-5)|d r|^2\\
   &\geq (k\delta-\frac{1}{2})r^2|A_N|^2+1-\frac{1}{4}k\\
   &\geq \frac{4-k}{4}
\end{align*}
where $\frac{3}{2}\leq k\leq 2$ and $\delta>\frac{1}{3}$. For simplicity, we take $k=2.$
Eventually by definition of $h$ we have
\begin{align*}
    \int_{\Sigma}2|\tilde{\nabla}^\Sigma\psi|^2&\geq \int_\Sigma (\frac{1}{4}-\tilde{K}_\Sigma)\psi^2u^2+\int_\Sigma \frac{1}{2}\left(\frac{1}{2}+h^2-2|\tilde{\nabla} h|\right)\psi^2\\
    &\geq \int_\Sigma (\frac{1}{4}-\tilde{K}_\Sigma)\psi^2.
\end{align*}
We then follow the idea of \cite[Lemma 16] {chodoshlisoapbubble} (see also \cite[Theorem 4.3]{hong24} for required modifications) to show that
\[\operatorname{diam}(\Sigma)\leq 4\pi.\]
Let $\psi=1,$ we have
\[\frac{1}{4}|\Sigma|\leq \int_\Sigma \tilde{K}_\Sigma=2\pi \chi(\Sigma)\leq 4\pi,\]
where $\chi(\Sigma)$ is the Euler characteristic of $\Sigma.$
Thus the area of $\Sigma$ is bounded by $16\pi.$
\end{proof}

\subsection{$\mu$-bubble in $M^4\subset \mathbb{R}^5$}%
\label{sec:_r_5_}
We assume $n=4$ and use the following notation.
We write $\left\{ e_1,\cdots ,e_n \right\}$ to be the orthonormal basis of $M$ under the induced metric $g$, and denote $\left\{ \tilde{e}_i=re_i \right\}_{i=1}^n$ to be the corresponding orthonormal basis of $M$ under the metric $\tilde{g}$.
In particular, $\tilde{\nu}$ denotes the unit normal of $M$ under the metric $\tilde{g}$.
Recall the bi-Ricci curvature of $M$ under metric $g$ is defined by
\[
	\operatorname{biRic}^M_{g}(u,v)=\Ric_g^M(u,u)+\Ric_g^M(v,v)-R_g(u,v,u,v),
\]
for any $g(u,u)=g(v,v)=1$ and $g(u,v)=0$.

\begin{proposition}\label{r2A2lowerboundinR5}
	
	We have
\[
	\frac{5}{6}r^2|A_N|^2\ge 7-5|dr|^2-\operatorname{biRic}^N_{\tilde{g}}(\tilde{e}_1,\tilde{\nu}).
\]
\end{proposition}
\begin{proof}
	Denote $e_2=\nu.$
Recall that we have (see \cite[(3.1)]{chodoshliR5})
\begin{align*}	r^2(A_{11}^2+A_{22}^2&+A_{11}A_{22}+\sum_{i>1} A_{1i}^2+\sum_{j>2} A_{2j}^2)+\left< \vec{x},\nu \right> (A_{11}+A_{22})\\=&10-7|dr|^2-(dr(e_1)^2+dr(e_2)^2)-\operatorname{biRic}^N_{\tilde{g}}(\tilde{e}_1,\tilde{\nu})\\
	\geq& 10-8|dr|^2-\operatorname{biRic}^N_{\tilde{g}}(\tilde{e}_1,\tilde{\nu}).
\end{align*}
Use the Cauchy-Schwarz inequality, we have
\[
	\left< \vec{x},\nu \right> (A_{11}+A_{22})\le 3dr(\nu)^2+\frac{(A_{11}+A_{22})^2}{12}r^2.
\]
We claim the following inequality
\[
	A_{11}^2+A_{22}^2+A_{11}A_{22}+\frac{(A_{11}+A_{22})^2}{12}\le \frac{5}{6}|A|^2.
\]
Indeed, by minimality of $N$ and the arithmetic inequality,
\begin{align*}
    \frac{5}{6}|A|^2\geq& \frac{5}{6}(A_{11}^2+A_{22}^2)+\frac{5}{6}(A_{33}^2+A_{44}^2)\\
    \geq &\frac{5}{6}(A_{11}^2+A_{22}^2)+\frac{5}{6}\frac{(A_{11}+A_{22})^2}{2}\\
    \geq & \frac{13}{12}(A_{11}^2+A_{22}^2)+\frac{7}{6}A_{11}A_{22}.
\end{align*}
This proves the claim.
By the fact that
\[
	1=dr(\nu)^2+|dr|^2,
\]
we obtain
\[
	\frac{5}{6}r^2|A|^2\ge 7-5|dr|^2-\operatorname{biRic}^N_{\tilde{g}}(\tilde{e}_1,\tilde{\nu}).
\]
\end{proof}

\begin{theorem}\label{diameterandvolumeboundn=4}
    Let $(N_0,\tilde{g})$ be a compact subset of $(N,\tilde{g})$ with connected boundary.  Suppose there exists $p\in N_0$ such that $d_{\tilde{g}}(p,\partial N_0)>4\sqrt{3}\pi. $ Then there exists a connected relative open set $\Omega$ containing $\partial N_0$, $\Omega\subset B_{4\sqrt{3}\pi}^{\tilde{g}}(\partial N_0)$, such that each connected component of $\partial \Omega\setminus \partial N_0$ has diameter bounded by $2\sqrt{3}\pi$ and volume bounded by $24\sqrt{3}|\mathbb{S}^3|$ under the metric $\tilde{g}$.
\end{theorem}
\begin{proof}
The proof follows the same lines as in the proof of Theorem \ref{diameterandvolumeboundn=3}. So we have
     \begin{align}
       \int_{\Sigma}\frac{4}{4-k}|\tilde{\nabla}^\Sigma\psi|^2&\geq \int_\Sigma \left(-k \frac{\tilde{\Delta}_N u}{u}+
|\tilde{A}_{\Sigma}|^2+\tilde{\Ric}^N(\tilde{\nu},\tilde{\nu})-\frac{1}{2}\tilde{H}_\Sigma^2\right)\psi^2 \nonumber\\
     \label{inequalityfordiameterandvolume}       &+\int_\Sigma \left(\frac{1}{2}h^2+(\tilde{\nabla}^N_{\nu}h)\right)\psi^2. 
    \end{align}
    We need to deal with curvature term
    \[C:=-k \frac{\tilde{\Delta}_N u}{u}+
|\tilde{A}_{\Sigma}|^2+\tilde{\Ric}^N(\tilde{\nu},\tilde{\nu})-\frac{1}{2}\tilde{H}_\Sigma^2+\tilde{\operatorname{Ric}}^\Sigma(\tilde{e}_1,\tilde{e}_1).\]
By the Gauss equation, it is equal to
\begin{align*}
    &-k \frac{\tilde{\Delta}_N u}{u}+\tilde{\operatorname{biRic}}^N(\tilde{e}_1,\tilde{\nu})+|\tilde{A}_\Sigma|^2+\tilde{H}_\Sigma \tilde{A}^\Sigma_{11}-\sum_{i=1}^3(\tilde{A}_{i1}^\Sigma)^2-\frac{1}{2}\tilde{H}_\Sigma^2\\
    &\geq -k \frac{\tilde{\Delta}_N u}{u}+\tilde{\operatorname{biRic}}^N(\tilde{e}_1,\tilde{\nu}).
\end{align*}
On the one hand,  by \eqref{conformaljacobi}
\[-k \frac{\tilde{\Delta}_N u}{u}=k(\delta r^2|A_N|^2-4+3|dr|^2).\]
On the other hand, by Proposition \ref{r2A2lowerboundinR5}
\[\tilde{\operatorname{biRic}}^N(\tilde{e}_1,\tilde{\nu})\geq 7-5|dr|^2-\frac{5}{6}r^2|A_N|^2.\]

Thus, if we take $k=\frac{5}{3}$,
\begin{align*}
    C&\geq (k\delta-\frac{5}{6})r^2|A_N|^2+7-4k+(3k-5)|dr|^2\geq \frac{1}{3}.
\end{align*}
Hence, \eqref{inequalityfordiameterandvolume} turns to
\begin{align*}
    \frac{12}{7}\int_\Sigma |\tilde{\nabla}^\Sigma\psi|^2&\geq \int_\Sigma(\frac{1}{6}-\tilde{\operatorname{Ric}}^\Sigma(\tilde{e}_1,\tilde{e}_1))\psi^2+\int_\Sigma\frac{1}{2}\left(\frac{1}{3}+h^2+2(\tilde{\nabla}^N_\nu h)\right)\psi^2\\
    &\geq \int_\Sigma(\frac{1}{6}-\tilde{\operatorname{Ric}}^\Sigma(\tilde{e}_1,\tilde{e}_1))\psi^2.
\end{align*}
The construction of $h$ is similar to the one in Theorem \ref{diameterandvolumeboundn=3} so that $1/3+h^2-2|\tilde{\nabla} h|\geq 0$. Then by Shen-Ye's work \cite{shenyingyerugang}, the diameter of $\Sigma$ is bounded above by $2\sqrt{3}\pi.$ The volume bound of $\Sigma$ follows from Chodosh-Li-Minter-Stryker \cite{chodoshliR5} and Antonelli-Xu \cite{antonellixu}.
\end{proof}



\subsection{Volume growth: Proof of Proposition \ref{thm:deriveAreaGrowth}}
\label{sec:areaGrowth}
In this part, we briefly explain the proof of \ref{prop:itEuclidean} of Proposition \ref{thm:deriveAreaGrowth} (cf. \cite[Theorem 2.3]{chodoshliR4anisotropic}). For \ref{prop:itGeodesic}, we just need to change the distance function used in the arguments to be intrinsic. Let $n=3$ or $4$.

Fix a large $R>0$.
Since $M$ has finite ends, $M\setminus B^{n+1}_{e^{4\sqrt{3}\pi}R}(0)$ has only finitely many unbounded components $\{E_1,\cdots,E_k\}$ where $k$ is the number of ends of $M$. Denote $M'$ the connected component of $M\setminus\{E_1,\cdots,E_k\}$ containing $0$.
By Theorem \ref{diameterandvolumeboundn=3} and \ref{diameterandvolumeboundn=4}, we can always find $M_0\subset M'$ with $d_{\tilde{g}}(\partial M_0,\partial M')\leq 4\sqrt{3}\pi$. The simply-connectedness implies that the number of component of $\partial M_0$ is bounded by $k$, thus $|\partial M_0|_{\tilde{g}}\leq k C$ and each component of $\partial M_0$ has intrinsic $\tilde{g}$-diameter $\leq 4\pi.$ Let $C_1=|\mathbb{B}^n||\mathbb{S}^{n-1}|^{-\frac{n}{n-1}}$.
By Lemma 6.2 (1) in \cite{chodoshliR4anisotropic}, we know that
\[\tilde{M}_R\subset M_0\subset M',\]
where $\tilde{M}_R$ is the connected component of $M\cap B^{n+1}_R(0)$ containing $0$.
Then by isoperimetric inequality for minimal hypersurfaces in Euclidean space (\cite{brendleisoperimetric}),
\[|\tilde{M}_R|\leq |M_0|\leq C_1|\partial M_0|_g^{\frac{n}{n-1}} \leq C_1(e^{4\sqrt{3}\pi}R)^n|\partial M_0|_{\tilde{g}}^\frac{n}{n-1}\leq C(n,k)R^n\]
completing the proof.



\section{Critical case for $n=3$}%
\label{sec:critical_case}
In this section, we characterize the $\frac{1}{3}$-stability of the $3$-dimensional catenoid.

\begin{proof}
[Proof of Theorem \ref{thm:catenoidMain}]
Let $\delta=\frac{n-2}{n}$ and $\beta=\frac{n-2}{n-1}.$
	We suppose that $M\subset \left\{ |x_{n+1}|<\Lambda \right\}$ and $\mathcal{H}^{n}(B_R^M(0))\le \Lambda R^n$ for every $R>0$ with $0 \in M$ for some positive real number $\Lambda$.
	We choose $\varphi \rightarrow (g^\beta-k)^+\varphi$ in $\delta$-stability inequality to obtain (see previous computation in Proposition \ref{prop:Alpha})
	\begin{align*}
	    \int_{ R_k} g^{2\beta-2}|A|^2\varphi^2&\left( 1-\frac{k}{g^\beta} \right) \beta \left( 1-\beta +\frac{1}{1-g^2} \right)E\\&\le \int_{ R_k} g^{2\beta}|\nabla \varphi|^2+\delta k g^{\beta}|A|^2 \varphi^2\left(1-\frac{k}{g^\beta}\right),
	\end{align*}
where $E:=\frac{n-1}{n}|A|^2(1-g^2)-|\nabla g|^2\geq 0$, and $R_k:=\{x\in M:g^\beta(x)>k\}$.
	Taking $k\rightarrow 0^+$, we have
	\[
		\int_{ M\cap \left\{ g>0 \right\}} g^{2\beta-2}|A|^2\varphi^2 \beta (1-\beta+\frac{1}{1-g^2})E\le \int_{M} g^{2\beta}|\nabla \varphi|^2.
	\]
	Then, we have
	\[
		\int_{ M\cap \left\{ g>0 \right\}} g^{2\beta-2}|A|^2\varphi^2E\le C\int_{M} g^{2\beta}|\nabla \varphi|^2,
	\]
	for some $C=C(n)$.

	We choose $X=\varphi x_{n+1}e_{n+1}$ in the first variation formula where $\varphi$ is a non-negative smooth function on $M$ with compact support in $B^M_{2R}(0)$, equal to $1$ in $B^M_R(0)$, and $|\nabla \varphi|\le \frac{2}{R}$.
	Then, we have
	\[
		\int_{ B^M_R(0) } g^2\le \frac{4}{R^2} \int_{  B^M_{2R}(0)\backslash B^M_R(0)} |x_{n+1}|^2\le CR^{n-2},
	\]
	for some $C=C(n,\Lambda)$,
	since $|x_{n+1}|$ is uniformly bounded.
	Now, we focus on the case $n=3$, and $\beta=\frac{1}{2}$.
	For a fixed positive integer $m$ and radius $R_0$, we choose the logarithmic test function $\varphi$ defined by
	\[
		\varphi(p)=
		\begin{cases}
		1, & \mathrm{dist}_M(0,p)\le R_0, \\
		1-\frac{1}{m}\ln (\frac{\mathrm{dist}_M(0,p)}{R_0}), & R_0\le \mathrm{dist}_M(0,p)\le R_0 e^m, \\
		0, & \mathrm{dist}(0,p)\ge R_0 e^m,
		\end{cases}
	\]
	where $\mathrm{dist}_M(\cdot,\cdot)$ denotes the intrinsic distance function on $M$.
	Then, 
	\begin{align*}
		&\int_{  \left\{ g>0 \right\}\cap B^M_{R_0}(0)} g^{-1}|A|^2E
            \le C\sum_{j=1 }^{m-1} \int_{  B^M_{R_0 e^{j+1}}(0)\backslash B^M_{R_0 e^j}(0)} g|\nabla \varphi|^2\\
		\le{}& C \sum_{j=1}^{m-1} \frac{1}{m^2R_0^2 e^{2j}}\int_{  B^M_{R_0 e^{j+1}}(0)\backslash B^M_{R_0 e^j}(0)} g\\
		 \le{}& C\sum_{j=1}^{m-1} \frac{1}{m^2 R_0^2e^{2j}}\left(\int_{  B^M_{R_0 e^{j+1}}(0)\backslash B^M_{R_0 e^j}(0)} g^2\right)^{\frac{1}{2}}(\mathcal{H}^n( B^M_{R_0 e^{j+1}}(0)\backslash B^M_{R_0 e^j}(0)))^{\frac{1}{2}}\\
		 \le{}& C \sum_{j =1}^{m} \frac{1}{m^2 R_0^2 e^{2j}} (R_0e^{j+1})^{\frac{1}{2}} (R_0e^{j+1})^{\frac{3}{2}}\\={}&\frac{C}{m}
	\end{align*}
	for some constant $C=C(n)$.
	Now, we can take $m\rightarrow \infty$ to obtain $E=0$ whenever $g\neq 0$ and $|A|\neq 0$.
	By the lemma presented below, we know if $M$ is not flat, $M$ must be a catenoid or $M=M' \times \mathbb{R} \subset \mathbb{R}^n\times \mathbb{R}$, where the second case is impossible by our boundedness assumption. Here, the notation $M=M'\times \mathbb{R}\subset \mathbb{R}^n\times \mathbb{R}$ means $M$ can be written as a product of $M'$ with the last coordinate and $M'$ is a complete immersed two-sided minimal hypersurface in $\mathbb{R}^n$.
\end{proof}
\begin{lemma}
	
	Let $M$ be a non-flat complete two-sided minimal hypersurface in $\mathbb{R}^{n+1}$.
	If $E\equiv 0$ and $M$ cannot be written as $M'\times \mathbb{R}\subset \mathbb{R}^{n}\times \mathbb{R}$, then $M$ must be a catenoid.
\end{lemma}
\begin{proof}
	Since $M$ is non-flat, we can select a non-empty subset $U\subset W$ such that $g(x)> 0$, $|A|(x)>0$.
	We choose a basis $\left\{ \tau_1,\cdots ,\tau_n \right\}$ of $T_xM$ such that $e_{n+1}^\top=|e_{n+1}^\top |\tau_1 $.
	Then, we know
	\begin{align*}
		\left|\nabla g\right|^2={}&\sum_{j=1}^n\left|\frac{\partial g}{\partial \tau_j}\right|^2=\sum_{j =1}^{n}\frac{(\nu\cdot e_{n+1})^2}{1-(\nu\cdot e_{n+1})^2} \left( \frac{\partial (\nu\cdot e_{n+1})}{\partial \tau_j} \right)^2\\={}&\sum_{j =1}^{n}\frac{(\nu\cdot e_{n+1})^2}{|e_{n+1}^\top |^2} (A(\tau_j,\tau_1)|e^\top _{n+1}|)^2\\
		 ={}& \sum_{j =1}^{n}(\nu\cdot e_{n+1})^2A_{1j}^2.
	\end{align*}
	Here, we write $A_{ij}=A(\tau_i,\tau_j)$.
	We write $E'=\frac{n-1}{n}|A|^2-\sum_{j =1}^{n}A_{1j}^2$.
	Then
	\begin{align}
		E'={} & \sum_{j=2 }^{n} \frac{n-2}{n}A_{1j}^2+\sum_{1<i<j }^{n} \frac{2n-2}{n}A_{ij}^2+\sum_{j=2 }^{n}\frac{n-1}{n}A_{jj}^2-\frac{1}{n}A_{11}^2\nonumber \\
		 ={}& \sum_{j=2 }^{n} \frac{n-2}{n}A_{1j}^2+\sum_{1<i<j }^{n} \frac{2n-2}{n}A_{ij}^2+\frac{1}{n}\left( (n-1)\sum_{j=2 }^{n}A_{jj}^2-\left( \sum_{j=2 }^{n}A_{jj} \right)^2 \right).
		 \label{eq:pfEprime}
	\end{align}
	If $(\nu\cdot e_{n+1})\neq 0$ at some open subset $U'\subset U$, then we have $E'=0$ on $U'$.
	From identity \eqref{eq:pfEprime}, we know $E'\equiv 0$ if and only if
	\[
		(n-1)A_{11}=-A_{jj},\quad \text{ for }2\le j\le n,\quad A_{ij}= 0\quad \text{ for }i\neq j.
	\]
	By a result of do Carmo and Dajczer \cite{doCarmo1983RotationHyper}, we know $M\cap U'$ is part of a catenoid.
	By the unique continuation theorem, we know $M$ is a catenoid.
	At last, if $(\nu\cdot e_{n+1})\equiv 0$ on $U$, we know $e_{n+1}$ is a tangent vector field on $U$ and hence $M\cap U$ is a product hypersurface locally.
	Again by the unique continuation theorem, we know $M$ is a product hypersurface and it is invariant under the translation along direction $e_{n+1}$.
\end{proof}

\bibliographystyle{alpha}
\bibliography{mybib,references}
\end{document}

%% file: head.tex
\usepackage{amsmath}
\usepackage{amssymb}
\usepackage{amsthm}
\usepackage{hyperref}
\usepackage[noabbrev,capitalize]{cleveref}
\usepackage{mathrsfs}
\usepackage{mathtools}
\usepackage{slashed}
\usepackage{tikz-cd}
\usepackage[shortlabels]{enumitem}

\setenumerate{label=(\roman*)}

\DeclareMathOperator{\Ric}{Ric}

\newcommand{\lp}{\langle}
\newcommand{\rp}{\rangle}

\newcommand{\mA}{\mathcal{A}}

\newcommand{\mH}{\mathcal{H}}

\def\sideremark#1{\ifvmode\leavevmode\fi\vadjust{\vbox to0pt{\vss
 \hbox to 0pt{\hskip\hsize\hskip1em
 \vbox{\hsize3cm\tiny\raggedright\pretolerance10000
 \noindent #1\hfill}\hss}\vbox to8pt{\vfil}\vss}}}

\newtheorem{theorem}{Theorem}[section]
\newtheorem{proposition}[theorem]{Proposition}
\newtheorem{lemma}[theorem]{Lemma}
\newtheorem{corollary}[theorem]{Corollary}

\theoremstyle{definition}

\newtheorem{conjecture}[theorem]{Conjecture}

\theoremstyle{remark}
\newtheorem{remark}[theorem]{Remark}

\numberwithin{equation}{section}

%% file: main.bbl
\begin{thebibliography}{CLMS24}

\bibitem[All72]{Allard1972}
William~K. Allard.
\newblock On the first variation of a varifold.
\newblock {\em Annals of mathematics}, 95:417, 1972.

\bibitem[AX24]{antonellixu}
Gioacchino Antonelli and Kai Xu.
\newblock New spectral {B}ishop-{G}romov and {B}onnet {M}yers theorems and applications to isoperimetry.
\newblock {\em arXiv:2024.08918}, pages 1--23, 2024.

\bibitem[Bel23]{bellettini}
Costante Bellettini.
\newblock Extensions of {S}choen--{S}imon--{Y}au and {S}choen--{S}imon theorems via iteration à la {D}e {G}iorgi.
\newblock {\em arXiv: 2310.01340}, 2023.

\bibitem[Bre21]{brendleisoperimetric}
Simon Brendle.
\newblock The isoperimetric inequality for a minimal submanifold in {E}uclidean space.
\newblock {\em J. Amer. Math. Soc.}, 34(2):595--603, 2021.

\bibitem[CD83]{doCarmo1983RotationHyper}
M.~Do Carmo and M.~Dajczer.
\newblock Rotation hypersurfaces in spaces of constant curvature.
\newblock {\em Transactions of the American Mathematical Society}, 277(2):685, June 1983.

\bibitem[CL21]{chodoshliR4}
Otis Chodosh and Chao Li.
\newblock Stable minimal hypersurfaces in $\mathbb{R}^4$.
\newblock {\em arXiv:2108.11462, To appear in Acta Mathematica}, 2021.

\bibitem[CL23]{chodoshliR4anisotropic}
Otis Chodosh and Chao Li.
\newblock Stable anisotropic minimal hypersurfaces in $\mathbb{R}^4$.
\newblock {\em Forum Math. Pi}, 11(e3):1--22, 2023.

\bibitem[CL24]{chodoshlisoapbubble}
Otis Chodosh and Chao Li.
\newblock Generalized soap bubbles and the topology of manifolds with positive scalar curvature.
\newblock {\em Ann. of Math. (2)}, 199(2):707--740, 2024.

\bibitem[CLMS24]{chodoshliR5}
Otis Chodosh, Chao Li, Paul Minter, and Douglas Stryker.
\newblock Stable minimal hypersurfaces in $\mathbb{R}^5$.
\newblock {\em arXiv:2401.01492}, 2024.

\bibitem[CLS24]{chodosh-li-stryker}
Otis Chodosh, Chao Li, and Douglas Stryker.
\newblock Complete stable minimal hypersurfaces in positively curved 4-manifolds.
\newblock {\em to appear in J. Eur. Math. Soc}, 2024.

\bibitem[CM04]{Colding-Minicozzi-The-space-of-emb...II}
Tobias~H. Colding and William~P. Minicozzi, II.
\newblock The space of embedded minimal surfaces of fixed genus in a 3-manifold. {II}. {M}ulti-valued graphs in disks.
\newblock {\em Ann. of Math. (2)}, 160(1):69--92, 2004.

\bibitem[CM11]{Colding2011}
Tobias~H. Colding and William~P. Minicozzi, II.
\newblock {\em A course in minimal surfaces}, volume 121.
\newblock American Mathematical Soc., 2011.

\bibitem[CM16]{chodoshmaximojdg}
Otis Chodosh and Davi Maximo.
\newblock On the topology and index of minimal surfaces.
\newblock {\em J. Differential Geom.}, 104(3):399--418, 2016.

\bibitem[CMR24]{catino}
Giovanni Catino, Paolo Mastrolia, and Alberto Roncoroni.
\newblock Two rigidity results for stable minimal hypersurfaces.
\newblock {\em Geom. Funct. Anal.}, 34(1):1--18, 2024.

\bibitem[CSZ97]{Cao-Shen-Zhu-infinitevolume}
Huai-Dong Cao, Ying Shen, and Shunhui Zhu.
\newblock The structure of stable minimal hypersurfaces in {${\bf R}^{n+1}$}.
\newblock {\em Math. Res. Lett.}, 4(5):637--644, 1997.

\bibitem[CZ09]{chengzhouoneend}
Xu~Cheng and Detang Zhou.
\newblock Manifolds with weighted {P}oincar\'{e} inequality and uniqueness of minimal hypersurfaces.
\newblock {\em Comm. Anal. Geom.}, 17(1):139--154, 2009.

\bibitem[dCP79]{doCarmo-Peng}
M.~do~Carmo and C.~K. Peng.
\newblock Stable complete minimal surfaces in {${\bf R}^{3}$} are planes.
\newblock {\em Bull. Amer. Math. Soc. (N.S.)}, 1(6):903--906, 1979.

\bibitem[DGS65]{DeGigoriSingolarita1965removableSing}
Ennio De~Giorgi and Guido Stampacchia.
\newblock Sulle singolarit{\`a} eliminabili delle ipersuperficie minimali.
\newblock {\em Atti Accad. Naz. Lincei Rend. Cl. Sci. Fis. Mat. Natur.(8)}, 38:352--357, 1965.

\bibitem[FC85]{Fischer-Colbrie-On-complete-minimal}
Doris Fischer-Colbrie.
\newblock On complete minimal surfaces with finite {M}orse index in three-manifolds.
\newblock {\em Invent. Math.}, 82(1):121--132, 1985.

\bibitem[FCS80]{Fischer-Colbrie-Schoen-The-structure-of-complete-stable}
Doris Fischer-Colbrie and Richard Schoen.
\newblock The structure of complete stable minimal surfaces in {$3$}-manifolds of nonnegative scalar curvature.
\newblock {\em Comm. Pure Appl. Math.}, 33(2):199--211, 1980.

\bibitem[Fu09]{Fu2009deltaStable}
Haiping Fu.
\newblock On complete $\delta$-stable minimal hypersurfaces in $\mathbb{R}^{n+1}$.
\newblock {\em Far East Journal of Mathematical Sciences}, 1, 07 2009.

\bibitem[Hon24]{hong24}
Han Hong.
\newblock {CMC} hypersurface with finite index in hyperbolic space $\mathbb{H}^4$.
\newblock {\em arXiv: Differential Geometry}, 2024.

\bibitem[HY24]{hongyan24}
Han Hong and Zetian Yan.
\newblock Rigidity and nonexistence of {CMC} hypersurfaces in 5-manifolds.
\newblock {\em arXiv:2405.06867}, page~27, 2024.

\bibitem[Kaw88]{kawai}
Shigeo Kawai.
\newblock Operator {$\Delta-aK$} on surfaces.
\newblock {\em Hokkaido Math. J.}, 17(2):147--150, 1988.

\bibitem[Li17]{lichaoindexestimate}
Chao Li.
\newblock Index and topology of minimal hypersurfaces in {$\Bbb{R}^n$}.
\newblock {\em Calc. Var. Partial Differential Equations}, 56(6):Paper No. 180, 18, 2017.

\bibitem[Li20]{Li2020bernsteinCone}
Caiyan Li.
\newblock Bernstein theorems for minimal cones with weak stability.
\newblock {\em Mathematical Research Letters}, 27(1):227--241, 2020.

\bibitem[LR89]{lopezros}
Francisco~J. L\'{o}pez and Antonio Ros.
\newblock Complete minimal surfaces with index one and stable constant mean curvature surfaces.
\newblock {\em Comment. Math. Helv.}, 64(1):34--43, 1989.

\bibitem[Maz24]{mazet}
Laurent Mazet.
\newblock Stable minimal hypersurfaces in $\mathbb{R}^6$.
\newblock {\em arXiv:2405.14676}, 2024.

\bibitem[MPR06]{meeksperezros}
William~H. Meeks, III, Joaqu\'{\i}n P\'{e}rez, and Antonio Ros.
\newblock Liouville-type properties for embedded minimal surfaces.
\newblock {\em Comm. Anal. Geom.}, 14(4):703--723, 2006.

\bibitem[Pog81]{Pogorelov-stable}
A.~V. Pogorelov.
\newblock On the stability of minimal surfaces.
\newblock {\em Dokl. Akad. Nauk SSSR}, 260(2):293--295, 1981.

\bibitem[Ros06]{Ros-onesided-minimalsurface}
Antonio Ros.
\newblock One-sided complete stable minimal surfaces.
\newblock {\em J. Differential Geom.}, 74(1):69--92, 2006.

\bibitem[Sch77]{Schoen1977thesis}
Richard~M. Schoen.
\newblock {\em Existence and regularity theorems for some geometric variational problems}.
\newblock Stanford University, 1977.

\bibitem[Sim68]{Simons1968minimalCone}
James Simons.
\newblock Minimal varieties in {Riemannian} manifolds.
\newblock {\em Annals of Mathematics}, 88(1):62--105, July 1968.

\bibitem[Sim77]{Simon1977removableSing}
Leon Simon.
\newblock On a theorem of de {G}iorgi and {S}tampacchia.
\newblock {\em Mathematische Zeitschrift}, 155(2):199--204, 1977.

\bibitem[Sim83]{simon1983lectures}
Leon Simon.
\newblock {\em Lectures on geometric measure theory}.
\newblock Centre for Mathematical Analysis, Australian National University, Canberra, 1983.

\bibitem[SS81]{SchoenSimon1981Regularity}
Richard Schoen and Leon Simon.
\newblock Regularity of stable minimal hypersurfaces.
\newblock {\em Communications on Pure and Applied Mathematics}, 34(6):741--797, November 1981.

\bibitem[SSY75]{Schoen1975curvature}
Richard Schoen, Leon Simon, and Shing-Tung Yau.
\newblock Curvature estimates for minimal hypersurfaces.
\newblock {\em Acta Mathematica}, 134(0):275--288, 1975.

\bibitem[SY96]{shenyingyerugang}
Ying Shen and Rugang Ye.
\newblock On stable minimal surfaces in manifolds of positive bi-{R}icci curvatures.
\newblock {\em Duke Math. J.}, 85(1):109--116, 1996.

\bibitem[Tys89]{Tysk1989finitenessIndex}
Johan Tysk.
\newblock Finiteness of index and total scalar curvature for minimal hypersurfaces.
\newblock {\em Proceedings of the American Mathematical Society}, 105(2):429--435, 1989.

\bibitem[TZ09]{tamzhou}
Luen-Fai Tam and Detang Zhou.
\newblock Stability properties for the higher dimensional catenoid in {$\Bbb R^{n+1}$}.
\newblock {\em Proc. Amer. Math. Soc.}, 137(10):3451--3461, 2009.

\bibitem[Whi16]{brianwhitenote}
Brian White.
\newblock Introduction to minimal surface theory.
\newblock In {\em Geometric analysis}, volume~22 of {\em IAS/Park City Math. Ser.}, pages 387--438. Amer. Math. Soc., Providence, RI, 2016.

\bibitem[Wic08]{Wickramasekera2008regularityCompactMult2}
Neshan Wickramasekera.
\newblock A regularity and compactness theory for immersed stable minimal hypersurfaces of multiplicity at most 2.
\newblock {\em Journal of Differential Geometry}, 80(1):79--173, 2008.

\bibitem[Wic14]{Wickramasekera2014Regularity}
Neshan Wickramasekera.
\newblock A general regularity theory for stable codimension 1 integral varifolds.
\newblock {\em Annals of Mathematics}, 179(3):843--1007, May 2014.

\end{thebibliography}
